\documentclass{amsart}

\usepackage{amsmath}
\usepackage{amsthm}
\usepackage{amssymb}
\usepackage{amsfonts}
\usepackage[shortlabels]{enumitem}
\usepackage[bookmarks=true,hyperindex,pdftex,colorlinks,citecolor=red, linkcolor=blue]{hyperref}
\usepackage{centernot} %for negations
\usepackage{marginnote} % add margin notes
\usepackage{bbm}
\usepackage[all]{xy}
\usepackage{tikz}
\usetikzlibrary{arrows}
\usetikzlibrary{shapes,decorations}
\usetikzlibrary{positioning}

\newcommand{\Cross}{$\mathbin{\tikz [x=1.4ex,y=1.4ex,line width=.2ex, red] \draw (0,0) -- (1,1) (0,1) -- (1,0);}$}%

\DeclareMathOperator{\dist}{dist}                           % distance between sets
\DeclareMathOperator{\lspan}{span}                          % linear span
\DeclareMathOperator{\supp}{supp}                           % support
\DeclareMathOperator{\Lip}{Lip}                             % Lipschitz functions
\newcommand{\N}{\mathbb{N}}             % natural numbers
\newcommand{\R}{\mathbb{R}}             % real numbers
\newcommand{\C}{\mathbb{C}}             % complex numbers
\newcommand{\set}[1]{\left\{{#1}\right\}}                   % set by extension
\newcommand{\duality}[1]{\left<{#1}\right>}                 % dual action
\newcommand{\F}{\mathcal{F}}                                % Lipschitz-free space
\newcommand{\Free}{\mathcal{F}}                             % Lipschitz-free space
\newcommand{\restricted}{\mathord{\upharpoonright}}
\renewcommand{\restriction}{\mathord{\upharpoonright}}

\newcommand{\Orb}{\mathrm{Orb}} 

\def\<{\langle}
\def\>{\rangle}
\newcommand{\ep}{\varepsilon}

\theoremstyle{plain}
\newtheorem{theorem}{Theorem}[section]
\newtheorem{lemma}[theorem]{Lemma}
\newtheorem{corollary}[theorem]{Corollary}
\newtheorem{proposition}[theorem]{Proposition}

\theoremstyle{definition}
\newtheorem*{definition*}{Definition}
\newtheorem{definition}[theorem]{Definition}
\newtheorem{example}[theorem]{Example}
\newtheorem{question}{Question}

\theoremstyle{remark}
\newtheorem{remark}[theorem]{Remark}

\begin{document}

\title[On the dynamics of Lipschitz operators]{On the dynamics of Lipschitz operators}

\author[A. Abbar]{Arafat Abbar}

\author[C. Coine]{Cl\'ement Coine}

\author[C. Petitjean]{Colin Petitjean}

\address[A. Abbar]{LAMA, Univ Gustave Eiffel, UPEM, Univ Paris Est Creteil, CNRS, F--77447, Marne-la-Vall\'ee, France}
\email{arafat.abbar@univ-eiffel.fr}

\address[C. Coine]{Normandie Univ, UNICAEN, CNRS, LMNO, 14000 Caen, France}
\email{clement.coine@unicaen.fr}

\address[C. Petitjean]{LAMA, Univ Gustave Eiffel, UPEM, Univ Paris Est Creteil, CNRS, F--77447, Marne-la-Vall\'ee, France}
\email{colin.petitjean@u-pem.fr}

\date{} % uncomment to remove date from title

% ABSTRACT

\begin{abstract}
By the linearization property of Lipschitz-free spaces, any Lipschitz map $f : M \to N$ between two pointed metric spaces may be extended uniquely to a bounded linear
operator $\widehat{f} : \F(M) \to \F(N)$ between their corresponding Lipschitz-free spaces.
In this note, we explore the connections between the dynamics of Lipschitz self-maps $f : M \to M$ and the linear dynamics of their extensions $\widehat{f} : \F(M) \to \F(M)$. This not only allows us to relate topological dynamical systems to linear dynamical systems but also provide a new class of hypercyclic operators acting on Lipschitz-free spaces. 
\end{abstract}

% KEYWORDS
\subjclass[2010]{Primary 47A16, 54H20 ; Secondary 46B20, 54E35}
%\subjclass[2010]{Primary 46B20; Secondary 46B04, 54E50}
%47A16: Cyclic vectors, hypercyclic and chaotic operators
%46B20 (Geometry and structure of normed linear spaces)
%46B04 (Isometric theory of Banach spaces)
%54H20: Topological dynamics 
%54E35: Metric spaces, metrizability
%54E50 (Complete metric spaces)

\keywords{Chaoticity, Cyclicity, Hypercyclicity, Lipschitz-free space, Supercyclicity, Transitivity, Weakly mixing. }

\maketitle

% MAIN DOCUMENT

\section{Introduction}

%------------------- DYNAMICS --------------------------------

A \textit{topological dynamical system} is a pair $(M,f)$ where $M$ is a metric space and $f : M \to M$ is continuous map. In topological dynamics, it is often assumed that M is compact. A  \textit{linear dynamical system} is a pair $(X,T)$ where   $X$ is a  Banach space (or, more generally, a Fr\'echet space) and $T$ is a bounded linear operator on $X$.
We refer to \cite{GrPe} (and references therein) for an introduction to dynamical systems as well as for more details on the next notions. In what follows, the pair $(M,f)$ stands for a topological dynamical system, while $(X,T)$ denotes a linear dynamical system. Let $\N$ denote the set of positive integers and let $\N_0=\N  \cup \{0\}$.  
For any point $x$ in  $M$, \textit{the orbit of $x$ under $f$} is defined by
$$
\Orb(x,f):=\lbrace f^{n}(x):\, n\in\N_0  \rbrace.
$$
We will say that \textit{$f$ is hypercyclic} if it a has a dense orbit, that is, there exists $x \in M$ such that $\Orb(x,f)$ is dense in $M$; such an $x$ will be called a \textit{hypercyclic element} for $f$.
Next,
we say that $f$ is  \textit{topologically transitive} if, for each pair of nonempty open sets $U,V$ of $M$,  there exists $n \in \N_0 $ such that $f^n(U) \cap V \neq  \emptyset$.
It is known that if $M$ has no isolated point then any hypercyclic map is also topologically transitive \cite[Proposition 1.15]{GrPe}. Conversely, if $M$ is a separable Baire space then a topologically transitive map is hypercyclic (see the remark after \cite[Theorem 1.2]{Survey}). We will also consider the following stronger notions:
\begin{itemize}[leftmargin=*]
    \item $f$ is \textit{(topologically) mixing} if  for each pair of nonempty open sets $U,V$ of $M$  there exists $N \in \N_0$ such that for every $n \geq N$,  $f^n(U) \cap V \neq  \emptyset$;
    \item $f$ is \textit{(topologically) weakly mixing} if $f\times f$ is topologically transitive on $M\times M$;
    \item $f$ is \textit{Devaney chaotic} if it is topologically transitive and its set of periodic points is dense in $M$. We recall that $x$ is a periodic point of $f$ if there exists $n \in \N$ such that $f^n(x)=x$, and we will denote by $\mathrm{Per}(f)$ the set of all periodic points of $f$.
\end{itemize}
It is straightforward that 
$$\text{mixing } \implies \text{ weakly mixing } \implies \text{ topologically transitive.} $$
Moreover, for every bounded linear operator $T$ defined on a separable Banach space $X$ (see \cite{Survey, GrPe}):
\[T \text{ is Devaney chaotic} \implies T \text{ is weakly mixing} \implies T \text{ is hypercyclic}.\]
Then, we say that $T$ is \textit{supercyclic} whenever there exists a vector $x\in X$ whose projective orbit, i.e. the set
$$
\Orb(\mathbb{K}\,x,T):=\lbrace \lambda T^nx:\, \lambda\in\mathbb{K},\, n\in\N_0 \rbrace,
$$
is dense in $X$. Such a vector $x$ is called a \textit{supercyclic vector} for $T$. 
Finally, recall that $T$ is \textit{cyclic} if there exists a vector $x\in X$, called a \textit{cyclic vector} for $T$, such that the linear span of the orbit of $x$ under $T$ is dense in $X$. Clearly, the following chain of implications holds:
\[\text{Hypercyclicity }\Rightarrow \text{ Supercyclicity }\Rightarrow \text{Cyclicity}.\]

\medskip

One of the main objectives of this paper is to relate topological dynamical systems to linear dynamical systems. Such a connection have already been explored for instance in \cite[Corollary 2.9]{Feldman} where it is built a universal linear operator $T :X \to X$ in such a way that, for any compact metric space $M$ and any continuous map $f : M \to M$, there is an invariant compact set $K \subset X$
such that $T\restricted_K$ is topologically conjugate to $f$. In our work, we consider a different point of view since we relate topological dynamical systems to linear dynamical systems by taking advantage of the fundamental linearization property of Lipschitz-free spaces. Let us briefly introduce the latter class of Banach spaces along with the mentioned linearization property; a more detailed overview will be made in Subsection~\ref{Subsection-Lipfree}. 
\smallskip

Let $(M,d)$ be a metric space equipped with a distinguished point denoted by $0 \in M$. Following \cite{GoKa_2003}, the Lipschitz-free space over $M$, denoted by $\F(M)$, is the canonical predual of the real Banach space $\Lip_0(M)$ of Lipschitz maps from $M$ to $\R$, vanishing at $0$, and equipped with the norm (the best Lipschitz constant $\mathrm{Lip}(f)$ of $f$):
$$\displaystyle
\mathrm{Lip}(f) :=  \sup_{x \neq y \in M} \frac{|f(x)-f(y)|}{d(x,y)}.$$
More precisely, 
$$\F(M) := \overline{ \mbox{span}}^{\| \cdot  \|}\left \{ \delta(x) \, : \, x \in M  \right \} \subset \Lip_0(M)^*,$$
where $\delta(x)$ is the evaluation functional defined by $\<f,\delta(x)\> = f(x)$ for any $f\in \Lip_0(M)$. It is readily seen that $\delta : x \mapsto \delta(x) \in \F(M)$ is an isometry. We wish to point out that the class of Lispchitz-free spaces is a powerful tool which has been used in various fields of Mathematics for proving deep results (e.g. \cite{GoKa_2003}), simplifying some proofs (e.g. \cite{RosendalTalk}) and constructing counterexamples (e.g. \cite{AlbiacKalton}).
The following linearization property of Lipschitz-free spaces is the cornerstone of our study.

\begin{proposition} \label{diagramfree}
Let $M$ and $N$ be two pointed metric spaces. Let $f \colon M \to N$ be a Lipschitz map such that $f(0_M) = 0_N$. Then, there exists a unique bounded linear operator $\widehat{f} \colon \F(M) \to \F(N)$ with $\|\widehat{f}\|=\mathrm{Lip}(f)$ and such that the following diagram commutes:
$$\xymatrix{
    M \ar[r]^f \ar[d]_{\delta_{M}}  & N \ar[d]^{\delta_{N}} \\
    \F(M) \ar[r]_{\widehat{f}} & \F(N)
  }$$
\end{proposition}

In this paper, by \textit{Lipschitz operator} we mean any bounded linear operator $\widehat{f} : \F(M) \to \F(N)$ as defined in the previous proposition. 
\bigskip

A very natural and intriguing question is whether linear properties of $\widehat{f}$ can be characterised by properties on $f$, or vice versa. For instance, compact operators have been considered in \cite{vargas2, vargas}.  In this note, we choose to focus on the dynamical properties introduced above.  
More precisely, we are interested in the following general questions:
\begin{question}\label{Q1}
Assume that $f: M \to M$ has a given dynamical property, what can be said about $\widehat{f} : \Free(M) \to \Free(M)$? 
\end{question}
Or conversely:
\begin{question} \label{Q2}
Assume that $\widehat{f} : \Free(M) \to \Free(M)$ satisfies a given dynamical property, what can be said about $f: M \to M$?
\end{question}
Furthermore, another important motivation for exploring these questions is to provide a new class of hypercyclic linear operators. As we shall explain latter, for some metric spaces there is a good description of the associated Lipschitz-free spaces as $L^1(\mu)$ spaces. This allows us for instance to recover some well-known examples, such as backward or forward shift operators, but also to give some new hypercyclic operators acting on $L^1(\mu)$ spaces, therefore providing a different angle on the study of the linear dynamics on $L^1(\mu)$. Of course, even when Lipschitz-free spaces are not isomorphic to $L^1(\mu)$ spaces, they give rise to interesting examples of Banach spaces and therefore possibly interesting examples of hypercyclic linear operators.
\smallskip

To the best of our knowledge, these directions are rather new and not much explored. With respect to Question~\ref{Q1}, some answers are given by M. Murillo-Arcila and A. Peris in \cite[Theorem~2.3]{MuPe}. Indeed, they prove that if $T:X \to X$ is a bounded operator and $K\subset  X$ is an invariant set for $T$ such that $0\in K$ and $T \restricted_K$ is weakly mixing (mixing, weakly mixing and chaotic, respectively), then $T \restricted_{\overline{\lspan} K}$ is also weakly mixing (mixing, weakly mixing and chaotic, respectively). Since by the very definition of Lipschitz-free spaces we have $\overline{\lspan} \; \delta(M) = \F(M)$, as a direct consequence they could obtain that if a Lipschitz selfmap $f :M \to M$ is weakly mixing (mixing, weakly mixing and chaotic, respectively) then so is $\widehat{f}$ (see Example~2.4~(3) in \cite{MuPe}). As we shall explain later, the reverse implications are not true in general. In fact, we define in Example~\ref{Example2} a Lipschitz self-map $f : [0,1] \to [0,1]$ such that $\widehat{f}$ is mixing and Devaney chaotic while $f$ is not even topologically transitive. 
\bigskip

Let us now describe the content of the paper. In what follows, unless otherwise specified $f$ will stand for a base-point preserving Lipschitz mapping $f : M \to M$ and $\widehat{f}$ for its linearization obtained by Proposition~\ref{diagramfree}. We will first introduce below the notation as well as the main tools related to Lipschitz-free spaces which we will use throughout the paper. Next, we shall start our study in Section~\ref{Section-FO} by giving some properties which are preserved by the functor $f \mapsto \widehat{f}$. For instance, it is easy to see that a Lipschitz map $f : M \to N$ has a dense range if and only if $\widehat{f} : \F(M) \to \F(N)$  also has a dense range (Proposition~\ref{denserange}). Similarly, it is readily seen that
if the set of periodic points of $f$ is dense in $M$, then the set of periodic points of $\widehat{f}$ is dense in $\mathcal F(M)$ (the converse being false; see Proposition~\ref{PropPeriodic} and Example~\ref{ExamplePeriodic}).
Another observation is that a point $x \in M$ is a hypercyclic element for $f$ if and only if $\delta(x)$ is a cyclic vector for $\widehat{f}$ (Proposition~\ref{PropHCtoC}). In fact, if $\gamma$ is a hypercyclic vector for $\widehat{f}$ then $\gamma$ must be infinitely supported (Proposition~\ref{Prop-support}).
\smallskip

Then, it is well-known that a bounded linear operator is weakly mixing if and only if it satisfies the ``Hypercyclicity Criterion" (shortened HC, see Section~\ref{section_criteria} for more details). So one can use the connection between  $f$ and $\widehat{f}$ (that is the linearization property) to transfer the conditions on $\widehat{f}$ stated in the Hypercyclicity Criterion to metric conditions on $f$. Doing so, we obtain a criterion that we will call "\textit{Hypercylicity Criterion for Lipschitz operators}" (shortened HCL) which turns out to be very useful in a number of examples. Of course, if $\widehat{f}$ satisfies the HCL then $\widehat{f}$ satisfies the HC and therefore is hypercyclic (Theorem~\ref{HCriterion}). However the converse is not true in general as we will show in Example~\ref{ExampleHCnotHCL}. We also notice that 
if $M$ is a complete space without isolated points and if $f$ is weakly mixing, then $\widehat{f}$ satisfies the HCL (Theorem~\ref{Thm_WMimpliesHCL}). We summarize the above mentioned general relations in the following diagram.

\begin{center}
\begin{tikzpicture}
\node (M) at (0,0) {$\widehat{f}$ mixing};
\node (WM) at (3,0) {$\widehat{f}$ weakly mixing};
\node (HC) at (7,0) {$\widehat{f}$ satisfies the HC};
\draw[-implies,double equal sign distance] (M) -- (WM);
\draw[implies-implies,double equal sign distance] (WM) -- (HC);
\node (fM) at (0,-1.5) {$f$ mixing};
\node (fWM) at (3,-1.5) {$f$ weakly mixing};
\node (fHC) at (7,-1.5) {$\widehat{f}$ satisfies the HCL};
\node at (-0.5,-0.8) {\cite{MuPe}};
\node at (2.5,-0.8) {\cite{MuPe}};
\node at (6.6,-0.8) {\ref{HCriterion}};
\node at (4.8,-1.2) {\ref{Thm_WMimpliesHCL}};
\draw[-implies,double equal sign distance] (HC) to [bend left] node [midway,right]{\;\ref{ExampleHCnotHCL}} (fHC);
\node at (7.3,-0.8) {\Cross};
\draw[-implies,double equal sign distance] (fM) -- (fWM);
\draw[-implies,double equal sign distance] (fM) -- (M);
\draw[-implies,double equal sign distance] (fWM) -- (WM);
\draw[-implies,double equal sign distance] (fWM) -- (fHC);
\draw[-implies,double equal sign distance] (fHC) -- (HC);
\end{tikzpicture} 
\end{center}

Since every linear operator satisfying the Hypercyclicity Criterion is hypercyclic, an obvious question is whether the HCL implies that $f$ is topologically transitive or has a dense orbit. Unfortunately, this is not the case (see Example~\ref{ExampleHCLnotTransitive} or Example~\ref{Example2} for instance) and we do not know how to characterise Lipschitz maps satisfying the HCL in dynamical terms. This is actually the point where the theory probably becomes less obvious since many natural and tempting implications fail. For instance, $f$ having a dense orbit
does not necessarily imply that $\widehat{f}$ does so. In fact, $\widehat{f}$ might even not be supercyclic (see Example~\ref{ExamplefTTfhatNOTsuper}).
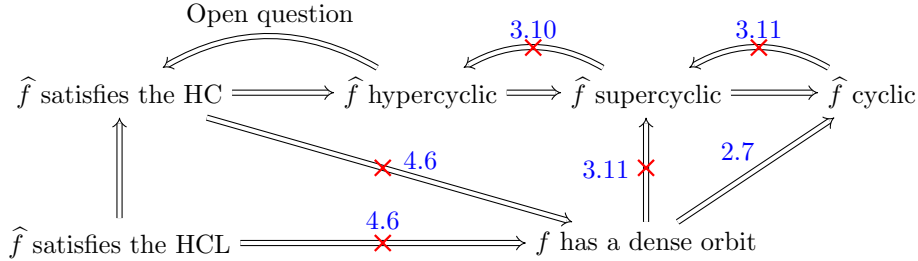
\begin{center}
\begin{tikzpicture}
\node (HC) at (0,0) {$\widehat{f}$ satisfies the HC};
\node (H) at (4,0) {$\widehat{f}$ hypercyclic};
\node (S) at (7,0) {$\widehat{f}$ supercyclic};
\node (C) at (10,0) {$\widehat{f}$ cyclic};
\draw[-implies,double equal sign distance] (HC) -- (H);
\draw[-implies,double equal sign distance] (H) -- (S);
\draw[-implies,double equal sign distance] (S) -- (C);
\node (fHC) at (0,-2) {$\widehat{f}$ satisfies the HCL};
\node (fTT) at (7,-2) {$f$ has a dense orbit};
\draw[-implies,double equal sign distance] (fHC) -- (HC);
\draw[-implies,double equal sign distance] (fHC) -- (fTT);
\node at (3.5,-1.7) {\ref{Example2}};
\node at (3.5,-2) {\Cross};
\draw[-implies,double equal sign distance] (fTT) -- node [midway,left]{\ref{ExamplefTTfhatNOTsuper}\;} (S);
\node at (7,-1) {\Cross};
\draw[-implies,double equal sign distance] (fTT) -- node [midway,above]{\ref{PropHCtoC}\;\;\;\;\;} (C);
\draw[-implies,double equal sign distance] (HC) -- (fTT);
\node at (4,-0.9) {\ref{Example2}};
\node at (3.5,-1) {\Cross};
\draw[-implies,double equal sign distance] (C) to [bend right] node [midway,above]{\ref{ExamplefTTfhatNOTsuper}} (S);
\node at (8.5,0.6) {\Cross};
\draw[-implies,double equal sign distance] (S) to [bend right] node [midway,above]{\ref{Counter-example}} (H);
\node at (5.5,0.6) {\Cross};
\draw[-implies,double equal sign distance] (H) to [bend right] node [midway,above]{Open question} (HC);
\end{tikzpicture} 
\end{center}

Most of our ``counter-examples" are built on discrete metric spaces $M$. This underlines that the structure of the metric space $M$ is as important as the Lipschitz self-map $f :M \to M$. For instance, there is no hypercyclic $\widehat{f}$ if $M$ is bounded and uniformly discrete (see Remark~\ref{UnifDiscrete}).
Leaving apart those pathological examples, 
one can obtain some positive results
by working on non-discrete metric spaces such as closed intervals in $\R$.  
Notably, we prove in Theorem~\ref{mainthminterval} that if $f : [a,b] \to [a,b]$ has a fixed point $c$ (considered to be the base-point of $M = [a,b]$) and is topologically transitive, then $\widehat{f}$ is weakly mixing. So the following implications hold for a base-point preserving Lipschitz self-map $f$ defined on a closed interval $M=[a,b]$:
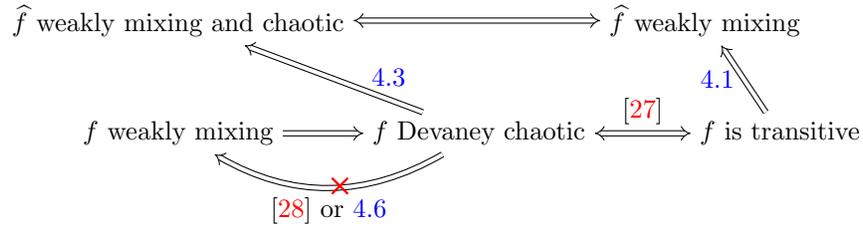
\begin{center}
\begin{tikzpicture}
\node (WMC) at (4,0) {$\widehat{f}$ weakly mixing and chaotic};
\node (WM) at (11,0) {$\widehat{f}$ weakly mixing};

\draw[implies-implies,double equal sign distance] (WMC) -- (WM);

\node (fTT) at (12,-1.5) {$f$ is transitive};
\node (fWM) at (4,-1.5) {$f$ weakly mixing};
\node (fDC) at (8,-1.5) {$f$ Devaney chaotic};

\draw[implies-implies,double equal sign distance] (fDC) -- node [midway,above]{\cite{VeBe}} (fTT);
\draw[-implies,double equal sign distance] (fWM) -- (fDC);
\draw[-implies,double equal sign distance] (fTT) -- node [midway,left]{\ref{mainthminterval}} (WM);
\draw[-implies,double equal sign distance] (fDC) --  (WMC);
\node at (6.8,-0.75) {\ref{CorDC}};

\draw[-implies,double equal sign distance] (fDC) to [bend left] node [midway,below]{\cite{Wang2011} or \ref{Example2}} (fWM); 
\node at (6.15,-2.2) {\Cross} ;
\end{tikzpicture} 
\end{center}

\subsection{Notation} \label{Subsection-Notation}
Let us now introduce the notation that will be used throughout this  paper. If $(M,d)$ is a metric space, we will denote by $B(x,r)$ the open ball of center $x\in M$ and radius $r>0$. When $E$ is a subset of $M$, we let $\dist(x,E) := \inf\{d(x,y) \; : \; y \in E\}$ be the distance from $x$ to $E$. If $(N, d')$ is another metric space and $f : M \to N$ is a Lipschitz map, then we let
$$\mathrm{Lip}(f) = \sup_{x\neq y} \dfrac{d'(f(x), f(y))}{d(x,y)}$$
be the smallest Lipschitz constant of $f$.
For a Banach space $X$, the unit ball of $X$ will simply be denoted by $B_X$ and its (topological) dual space by $X^*$.
If $Y$ is another Banach space, we will write $X \equiv Y$ if there exists an isometric isomorphism between $X$ and $Y$.
Finally, if $f : E \to F$ is a map between two sets and $U$ is a subset of $E$, $f\restricted_U$ will stand for the restriction of $f$ to $U$.

\subsection{Lipschitz-free spaces} \label{Subsection-Lipfree}
We wish to end this introduction by giving a more detailed introduction to Lipschitz-free spaces theory (for the proofs, we refer the reader to \cite{Weaver2} where the name Arens-Eells spaces is used instead). Consider a pointed metric space $(M,d)$ with distinguished point $0\in M$. For a \textbf{real} Banach space $X$, we denote by $\Lip_0(M,X)$ the vector space of Lipschitz maps from $M$ to $X$ satisfying $f(0)=0$. Then $\mathrm{Lip}(\cdot)$ is a norm on $\Lip_0(M,X)$, and equipped with that norm,
$\Lip_0(M,X)$ is a Banach space. When the range space is $\R$, we simply write $\Lip_0(M)$ instead of $\Lip_0(M,\R)$. Now recall that 
 the Lipschitz-free space over $M$ 
is the following subspace of $\Lip_0(M)^*$:
$$\F(M) := \overline{ \mbox{span}}^{\| \cdot  \|}\left \{ \delta(x) \, : \, x \in M  \right \},$$
where $\delta(x)$ is the functional defined by $\<f,\delta(x)\> = f(x)$ for every $f\in \Lip_0(M)$. It is readily seen that $\delta(x) \in \Lip_0(M)^*$ with $\|\delta(x)\| = d(x,0)$. The map $\delta_{M} \colon x \in M \mapsto \delta(x) \in \F(M)$ is actually an isometry which in turns implies that $\delta(M)$ is a closed subset of $\F(M)$ whenever $M$ is complete. In fact, if $\overline{M}$ is the completion of $M$ then $\F(M)$ and $\F(\overline{M})$ are linearly isometric. So, even when it is not precisely specified, we will always assume our metric spaces to be complete.
Notice also that $\F(M)$ is separable if and only if $M$ is so. 
\smallskip

Their most important application to non-linear geometry is certainly their universal extension  property: for every Banach space $X$, for every $f \in \Lip_0(M,X)$, the unique linear operator $\overline{f} \colon \F(M) \to X$ defined on $\lspan \delta(M)$ by 
$$\overline{f}\Big(\sum_{i=1}^n a_i \delta(x_i)\Big)= \sum_{i=1}^n a_i f(x_i) \in X $$
is continuous with $\|\overline{f}\|=\mathrm{Lip}(f)$. In other words, the map $\Phi \colon f\in \Lip_0(M,X) \mapsto \overline{f} \in \mathcal{L}(\F(M),X)$
 is an onto linear isometry. As a direct consequence (in the case $X=\R$) we obtain that $\F(M)^* \equiv \Lip_0(M)$. Moreover the weak$^*$ topology coincides with the topology of pointwise convergence on bounded sets of $\Lip_0(M)$.
\smallskip

Afterward, if $N \subset M$ with $0 \in N$ then $\F(N)$ can be
canonically isometrically identified with the subspace $\mathrm{span}\{\delta(x) : x \in N\}$ of $\F(M)$.
This is due to a well known McShane-Whitney theorem (see \cite[Theorem 1.33]{Weaver2} e.g.)
according to which every
real-valued Lipschitz function on $N$ can be extended to $M$ with the same Lipschitz
constant.

\begin{remark}
In linear dynamics, one often study operators defined on \textbf{complex} Banach spaces. 
Here we want to highlight the fact that, by construction, Lipschitz-free spaces are Banach spaces over $\mathbb R$. Nevertheless, one could build a complex version of Lipschitz-free spaces by following the same steps as we did above. That is, we may consider the complex Banach space $\Lip_0(M,X)$, where $X$ is a Banach space over $\C$ as well, and then the evaluation functionals $\delta(x) \in \Lip_0(M,\C)^*$ are defined in a same fashion. This leads to the complex version of the Lipschitz-free space 
$$ \F_{\C}(M) := \overline{ \mbox{span}}^{\| \cdot  \|}\left \{ \delta(x) \, : \, x \in M  \right \} \subset \Lip_0(M,\C)^*.$$
One can prove that the universal extension property works perfectly fine and thus provides $\F_{\C}(M)^* \equiv  \Lip_0(M,\C)$. Now one should be careful since some features of $\F(M)$ might not work equally well for $\F_{\C}(M)$ (for instance $\F_{\C}(N)$ may not be isometric but only isomorphic to a subspace of $\F(M)$). Up to our knowledge, the complex version of Lipschitz-free spaces have not been much studied in the literature (see the comments at pages 86 and 125 in \cite{Weaver2}). 
In our work, we claim that the results still hold if one replaces $\F(M)$ by $\F_{\C}(M)$.
\end{remark}

We now recall the fundamental linearization property of Lipschitz-free spaces (already stated in Proposition~\ref{diagramfree}), which is a direct consequence of the universal extension property presented above.  If $f \colon M \to N$ is a Lipschitz map such that $f(0_M) = 0_N$, then there exists a linear bounded operator $\widehat{f} \colon \F(M) \to \F(N)$ such that $\|\widehat{f}\|=\mathrm{Lip}(f)$ and which satisfies:
$$ \text{For any } \gamma = \sum_{i=1}^n a_i\delta_M(x_i) \in \F(M) , \quad \widehat{f}(\gamma) = \sum_{i=1}^n a_i \delta_N(f(x_i)).$$
We recall that such an operator $\widehat{f}$ will be called \textit{Lipschitz operator.}
\smallskip

In this paper, we will focus on Lipschitz self-maps $f : M \to M$ preserving the distinguished point and we will often require that $f$ is transitive. It is readily seen that if $f$ is transitive then its Lipschitz constant $\mathrm{Lip}(f) > 1$. 
Notice also that if $0$ is an isolated point in $M$, then there is no Lipschitz map $f:M \to M$ and $x\in M$ such that $f(0)=0$ and $\Orb(x,f)$ is dense in $M$ (and thus no hypercyclic $f :M \to M$).

\begin{remark} \label{UnifDiscrete}
We recall that any infinite-dimensional Banach space supports a hypercyclic operator \cite{Ansari}. Yet, for some metric spaces $M$ there is no hypercyclic Lipschitz operator. For instance, let $M$ be a countable separable pointed metric space and suppose that:
\begin{itemize}
    \item $M$ is  uniformly discrete, that is there exists $\theta >0$ such that $d(x,y)> \theta$ for every $x\neq y$;
    \item $M$ is  bounded, i.e., $\mathrm{rad}(M):=\underset{x\in M}{\sup}\,d(x,0)<+\infty$.
\end{itemize}
Then it is known \cite[Proposition~4.4]{Kalton04} that $\F(M)$ is linearly isomorphic to the Banach space $\ell_1(\N)$ of real sequences indexed by $\N$ whose series is absolutely convergent. However, every orbit under the action of $\widehat{f}$ is bounded, so $\widehat{f}$ cannot be hypercyclic:
\begin{eqnarray*}
\forall \gamma = \sum_{i=1}^{\infty} a_i \delta(x_i), \forall n \in \N, \quad \| \widehat{f}^n \gamma\| \leq \mathrm{rad}(M) \sum_{i=1}^{\infty} |a_i| \leq C \cdot \mathrm{rad}(M) \|\gamma\|.
\end{eqnarray*}
\end{remark}

\begin{remark}
A change of the base point in a metric space $M$ does not affect the isometric structure of the associated Lipschitz-free space. Indeed, if $b \in M$ is the new base point (instead of $0$), then $f \in \Lip_0(M) \mapsto f - f(b) \in \Lip_b(M)$ defines a linear and surjective isometry.
Moreover, it is easy to check that this operator is continuous with respect to the topology of pointwise convergence, which in turn implies that it is weak$^*$-to-weak$^*$ continuous. Therefore its preadjoint is a surjective isometry between $\F_b(M)$ and $\F(M)$, where $\F_b(M)$ is the Lipschitz-free space over $M$ with $b$ considered to be the distinguished point.

Now imagine that a Lipschitz self-map $f : M \to M$ admits two fixed points, say $p$ and $q$. One can consider $\widehat{f_p} : \F_p(M) \to \F_p(M)$ and $\widehat{f_q} : \F_q(M) \to \F_q(M)$ obtained by the linearization property of Lipschitz-free spaces. Let us denote by $T: \F_p(M) \to \F_q(M)$ the isometry described in the previous paragraph. Then, it is easy to check that $\widehat{f_q} = T \circ \widehat{f_p} \circ T^{-1}$. Therefore $\widehat{f_p}$ and $\widehat{f_q}$ are conjugate and they will enjoy the very same dynamical properties.
\end{remark}

To conclude this short introduction to Lipschitz-free spaces theory, we recall two famous examples and then discuss a more generic point of view.

\begin{example}
 In the sequel, $L^1 = L^1([0,1])$ denotes the real Banach space of integrable functions from $[0,1]$ to $\R$ (as usual quotiented by the kernel of $\|\cdot\|_1$).
\begin{enumerate}[itemsep=3pt]
\item ``$(M,d) = (\N,|\cdot|)$". The linear operator satisfying $T \colon \delta(n) \in \F(\N) \mapsto \sum_{i=1}^n  e_i \in \ell_1(\N)$ is an onto linear isometry (the sequence $(e_n)_n \subset \ell_1$ stands for the canonical unit vector basis of $\ell_1$).
\item ``$M =( [0,1],|\cdot|)$". The linear operator $T \colon \delta(t) \in \F([0,1]) \mapsto \mathbbm{1}_{[0,t]} \in L^1([0,1])$ is an onto linear isometry.
\end{enumerate}
\end{example}
More generally, we can see the two previous examples as particular cases of a more general theorem. Indeed, A. Godard gave a very explicit formula in \cite{Godard_2010} to prove that if $M$ is a subset of an $\R$-tree which contains all of its branching points, then $\F(M)$ is isometric to an $L^1(\mu)$ space. We recall that an $\R$-tree
is an arc-connected metric space $(M,d)$ with the property that there is a unique arc connecting any pair of points $x\neq y\in M$ and it moreover is isometric to the real segment $[0,d(x,y)]\subset\R$.
A point $x\in M$ is called a \emph{branching point} of $M$ if $M\setminus\set{x}$ has at least three connected components. 

In this paper we use Godard's formula in a number of examples. We will always give the definition of the isometries for convenience, but we will never prove that they are indeed surjective isometries. In fact, most of the time we apply Godard's formula to a countably branching tree of height 1, we state the explicit isometry here for future reference.  

\begin{proposition} \label{PropTreesl1}
Let $M = \N \cup \{0\}$ be equipped with the tree metric $d$ described below
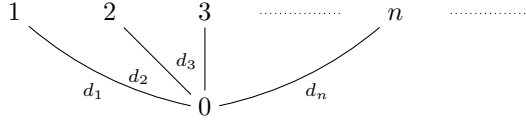
\begin{equation*}
\xymatrix @!0 @R=1pc @C=1.5pc {
1 && 2 && 3 & \ar@{.}[rr] &&& n & \ar@{.}[rr] && \\
\\
\\
& & & & \ar@/^/@{-}[lllluuu]^{d_1} \ar@{-}[lluuu]^{d_2} \ar@{-}[uuu]^{d_3} 0 \ar@/_/@{-}[rrrruuu]_{d_n}
}
\end{equation*}
That is, for every $n,m \in \N$, $d(n,0) =d_n > 0$ and $d(n,m) =d_n +d_m$. Then the linear map $\Phi : \F(M) \to \ell_1(\N)$ given by $\Phi(\delta(n)) = d(n,0)e_n = d_n e_n$ is a linear onto isometry. In particular, any Lipschitz operator $\widehat{f} : \F(M) \to \F(M)$ is conjugate to a bounded operator $T: \ell_1(\N) \to \ell_1(\N)$ such that $Te_n = d_{f(n)}d_n^{-1} e_{f(n)}$.
\end{proposition}

We refer the reader to the papers \cite{Godard_2010, AlPePr_2019} for more details on this topic.

\section{First observations} \label{Section-FO}

As we already mentioned, our aim is to study whether the arrow $f \mapsto \widehat{f}$ carries on some dynamical information. First, we note that having a dense range is preserved through this functor.

\begin{proposition} \label{denserange}
Let $M$ and $N$ be two pointed metric spaces and let $f :M \to N$ be a Lipschitz map such that $f(0_M) = 0_N$. Then, the range of $\widehat{f}$ is dense in $\F(N)$ if and only if the range of $f$ is dense in $N$.
\end{proposition}

\begin{proof} By the very definition of $\widehat{f}$, notice that $\widehat{f}(\lspan \delta(M)) = \lspan \delta(f(M))$.
\smallskip

\noindent $(\impliedby):$ If $f(M)$ is dense in $N$, then $\delta(f(M))$ is dense in $\delta(N)$ because the map $\delta$ is an isometry. Then, $\lspan \delta(f(M))$ is dense in $\overline{\lspan}\ \delta(N)  = \Free(N)$. Since $\lspan \delta(f(M)) = \widehat{f} (\lspan (\delta(M))) \subset \widehat{f}(\Free(M)),$ we get that $\widehat{f}(\Free(M))$ is dense in $\Free(N)$.
\smallskip

\noindent $(\implies):$ 
Assume that $f(M)$ is not dense in $N$ and let
$y\in N \setminus \overline{f(M)}$. Since $\dist\big(y,\overline{f(M)}\big):=\inf\big\{d(y,z) \; : \; z \in \overline{f(M)}\big\} >0$, we may define a Lipschitz map $g : N \to \mathbb{R}$ such that $g(y) = 1$ and $g(\overline{f(M)}) = \{0\}$ (such a map $g$ exists, see for instance the inf/sup-convolution formula \cite[Theorem 1.33]{Weaver2} to extend Lipschitz maps). In particular $g\in \Lip_0(N)$ and it is readily seen that $\langle g , \gamma \rangle = 0$ whenever $\gamma \in \overline{\widehat{f}(\F(M))}
$. Therefore, the fact that $\widehat{f}$ does not have a dense range follows from the next simple estimates:
\begin{align*}
 \dist\big(\delta(y) ,\overline{\widehat{f}(\F(M))}\big) 
 \geq \inf_{\gamma \in \overline{\widehat{f}(\F(M))}}\left| \left\langle \delta(y) - \gamma , \dfrac{g}{\|g\|} \right\rangle \right| = \frac{1}{\|g\|} >0.
\end{align*}
\end{proof}

\begin{corollary}\label{suphypdenserange}
Let $M$ be a pointed metric space and let $f :M \to M$ be a Lipschitz map such that $f(0) = 0$. If $\widehat{f}$ is supercyclic (or hypercyclic), then the range of $f$ is dense in $M$.
\end{corollary}

The forward shift  operator on  $\ell_1(\N)$ is cyclic, but its image is not dense in $\ell_1(\N)$. This allows us bellow to show that the cyclicity of $\widehat{f}$ does not imply that $f$ has a dense image in $M$. This also underlines the fact that the hypercyclicity of $f$ does not imply the supercyclicity of $\widehat{f}$. 

\begin{example} \label{4}
Let $f$ be the map defined on $M=\lbrace 1,2,3,...\rbrace\cup \lbrace0\rbrace$ by $f(0) = 0$ and $f(n)=n+1$ for every $n \in \N$. We equip $M$ with the tree metric $d$ given by: for all  $ n,m \geqslant 1$,   $d(n,0)=\frac{1}{n}$ and  $d(n,m)=d(n,0)+d(m,0)$.  
According to Proposition~\ref{PropTreesl1}, $\delta(n) \in \F(M) \mapsto \frac{1}{n} e_n \in \ell_1$ extends to a bijective linear isometry. In particular, $\widehat{f}$ is conjugate to the operator $T$ acting on $\ell_1$ by $Te_n = \frac{n}{n+1} e_{n+1}$.
Thus $\widehat{f}$ is cyclic while $f$ does not have a dense range. By Corollary $\ref{suphypdenserange}$, this implies that $\widehat{f}$ is not supercyclic.  
Notice also that $\Orb(1,f)$ is dense in $M$.
\end{example}

Nevertheless, the example above is somehow the only pathology that may occur, as this is shown by the next result.

\begin{proposition} \label{propcyclicdense}
If a Lispchitz operator $\widehat{f} : \F(M) \to \F(M)$ is cyclic, then either $f(M)$ is dense in $M$ or there exists $x \in M$ such that the range $f(M)$ is dense in $M\setminus \{x\}$.    
\end{proposition}

Before giving the proof, let us state the following simple facts which we will use throughout the section. 
\begin{lemma}\label{LemmaOrbit}
~~
\begin{enumerate}[itemsep = 3pt]
    \item For every $n \in \N$, $\widehat{f^n} = (\widehat{f})^n$.
    \item For every $x \in M$, $\Orb(\delta(x),\widehat{f}) = \delta(\Orb(x,f))$.
\end{enumerate}
\end{lemma}
\vspace{-0.12cm}

\begin{proof}[Proof of Proposition~\ref{propcyclicdense}]
Assume that $M\setminus \overline{f(M)}$ contains at least two points, say $x_1, x_2 \in M$. Set $E:=\lspan(\delta(x_1), \delta(x_2))$. Let $P : \Free(M) \to E$
be a continuous projection from $\Free(M)$ onto $E$ such that $P\restricted_{\overline{\lspan}\{\delta(f(M))\}}=0$.  If there exists $\gamma\in \Free(M)$ such that  $\lspan \Orb(\gamma,\widehat{f}) $ 
 is dense in $\Free(M)$, then 
 $$P\left(\lspan \Orb(\gamma,\widehat{f})  \right) = \lspan \left\lbrace P(\widehat{f}^n(\gamma)), n\geq 0 \right\rbrace$$
is dense in $E$.
However, notice that for any $n\geq 1$, $P(\widehat{f}^n(\gamma))=0$. Indeed, if $y\in M$ then $f^n(y)\not \in \{x_1, x_2\}$, so that $P(\widehat{f}^n(\delta(y)))=P(\delta(f^n(y))=0$. By linearity, this implies that $P(\widehat{f}^n(z))=0$ for any $z\in \lspan(\delta(M))$. By approximation and continuity of $P$, we get that $P(\widehat{f}^n(\gamma))=0$.
The latter implies that $P\left(\lspan \Orb(\gamma,\widehat{f})  \right) = \mathbb{R}.P(\gamma)$
 which is of dimension $1$ and hence cannot be dense in the $2-$dimensional space $E$. 
\end{proof}
Next, we deduce from Lemma~\ref{LemmaOrbit}~(1) the next proposition. 
\begin{proposition} \label{PropPeriodic}
Let $M$ be a pointed metric spaces and let $f :M \to M$ be a Lipschitz map such that $f(0) = 0$.
If the set of periodic points $\mathrm{Per}(f)$ of $f$ is dense in $M$, then the set of periodic points $\mathrm{Per}(\widehat{f})$ of $\widehat{f}$ is dense in $\Free(M)$.
\end{proposition}
\begin{proof}
Since $\mathrm{Per}(f)$ is dense in $M$, $\lspan \delta(\mathrm{Per}(f))$ is dense in $\F(M)$. Now if $\gamma = \sum_{i=1}^n a_i \delta(x_i) \in \lspan \delta(\mathrm{Per}(f))$, then for every $i \in \{1 , \ldots , n\}$ there exists $n_i \in \N$ such that $f^{n_i}(x_i)=x_i$. We define $n = \prod_{i=1}^n n_i$, notice that $f^{n}(x_i)=x_i$ for every $i \in \{1, \ldots , n\}$. We conclude the proof by using Lemma~\ref{LemmaOrbit}~(1) to show that 
$$ \duality{(\widehat{f})^{n} , \gamma} = \duality{\widehat{f^{n}} , \gamma} = \sum_{i=1}^n a_i \delta(f^{n}(x_i)) = \sum_{i=1}^n a_i \delta(x_i) = \gamma,$$
which implies that $\lspan \delta(\mathrm{Per}(f)) \subset \mathrm{Per}(\widehat{f})$.
\end{proof}

Another direct consequence of Lemma~\ref{LemmaOrbit}~(2) is that $\Orb(\delta(x),\widehat{f}) \subseteq \delta(M)$, and thus neither $\Orb(\delta(x), \widehat{f})$ nor $ \Orb(\mathbb R \,\delta(x),\widehat{f})$ (when $M\neq \{0,x\}$) can be dense in $\F(M)$. In other words, $\delta(x) \in \F(M)$ will never be a hypercyclic (or supercyclic) vector for the operator $\widehat{f}$. 
Nevertheless, $\delta(x) \in \F(M)$ may be a cyclic vector for $\widehat{f}$.

\begin{proposition} \label{PropHCtoC} 
Let $M$ be a metric space with non-isolated distinguished point $0\in M$, $f :M \to M$ be a Lipschitz map such that $f(0) = 0$, and let  $x\in M$. Then the following assertions are equivalent:
\begin{enumerate}[label={$(\arabic*)$}]
    \item $x$ is a hypercyclic element for  $f$.
    \item $\delta(x)$ is a cyclic vector for  $\widehat{f}$.
\end{enumerate}
\end{proposition}

\begin{proof}
(1) $\implies$ (2): Thanks to Lemma~\ref{LemmaOrbit}, $\Orb(\delta(x),\widehat{f}) = \delta(\Orb(x,f))$. So, if $\Orb(x,f)$ is dense in $M$ then $\lspan \delta(\Orb(x,f))$ is dense in $\F(M)$, which in turn implies that $\lspan \Orb(\delta(x),\widehat{f})$ is dense in $\F(M)$.

(2) $\implies$ (1): Assume that $\Orb(x,f)$ is not dense in $M\setminus \set{0}$. So there exists $y \neq 0 \in M$ and $\ep >0$ such that $ B(y,\ep) \cap \overline{\Orb(x,f)}   = \emptyset$. Then let $F \in \Lip_0(M)\equiv \F(M)^*$ be such that $F(y) >0$ and $\mathrm{supp}(F) \subseteq B(y,\ep)$ (for instance $z \in M \mapsto  d\big(z,B(y,C \ep)^c\big)$ for some small enough constant $C$ which ensures that $0 \not \in B(y, C\ep)$). Clearly, for every $\gamma \in \lspan \overline{\Orb(\delta(x), \widehat{f})}$,
 $\< F ,  \gamma\> =0$. However $\< F ,  \delta(y)\> >0$ which implies that 
$$\dist_{\F(M)}\left(\delta(y) , \lspan \Orb(\delta (x),\widehat{f}) \right)  = \dist_{\F(M)}\Big(\delta(y) , \lspan  \delta(\Orb(x, f)) \Big) > 0,$$ which means that  $\lspan \Orb(\delta(x),\widehat{f})$ is not dense in $\F(M)$.
\end{proof}

\subsection{Supports of supercyclic vectors}

As we  mentioned above,  evaluation functionals, that is elements with only one non-zero single point in their support, cannot be hypercylic or supercyclic vectors for some $\widehat{f}$. In fact,
we can say more: if $\gamma \in \F(M)$ is a hypercyclic (or supercyclic) vector for $\widehat{f}$, then $\gamma$ cannot be finitely supported.

\begin{definition}
Let $M$ be a pointed metric space. 
We say that $ \gamma\in \F(M)$ is finitely supported if
$$\gamma \in \mathrm{span}\lbrace \delta(x) : x \in M \rbrace.$$
The support of such a $\gamma$ is denoted by $\supp \gamma$ and is the smallest finite subset $F$ of $M$ which contains $0$ and
such that $\gamma \in \mathrm{span}\lbrace \delta(x) : x \in F \rbrace$.

More generally \cite{AP20, APPP_2019}, the support of any element $\gamma \in \F(M)$, also denoted by $\supp \gamma$, is the intersection of all closed subsets $K$ of $M$ such that $\gamma \in \F(K) \subset \F(M)$. It follows from \cite[Theorem~2.1]{APPP_2019} that $\gamma \in \F(\supp \gamma)$ and of course $\supp \gamma$ is the smaller closed subset with this property. 
\end{definition}

\begin{proposition} \label{Prop-support}
Let $M$ be an infinite complete metric space. If $\gamma \in \F(M)$ is a supercyclic vector for a Lipschitz operator $\widehat{f} : \F(M) \to \F(M)$, then $\gamma$ is infinitely supported.  
\end{proposition}

Indeed, any element in $\Orb(\gamma,\widehat{f})$ should have a support of cardinal less or equal to the one of $\gamma$. So our claim follows from the next result which was proved by R. Aliaga, C. Noûs, the third named author and A. Proch\'azka. We are deeply grateful to them for allowing us to include their result as well as its proof.

\begin{lemma}[\textbf{Aliaga, Noûs, Petitjean and Proch\'azka}] Let $M$ be a complete pointed metric space.
Let $FS_n(M) = \{\gamma \in \F(M) \; : \; | \supp \gamma | \leq n  \}$ be the set of finitely supported elements whose support contains at most $n$ points of $M$. Then $FS_n(M)$ is weakly closed.
\end{lemma}

The proof uses the notion of support introduced above. More precisely, we will need the following characterization (see \cite[Proposition 2.7]{APPP_2019}): Let $M$ be a complete pointed metric space and $\gamma \in \F(M)$.  Then $x \in M$ lies in the support of $\gamma$ if and only if for every open neighbourhood $U_x$ of $x$ there exists a function $f\in \Lip_0(M)$ whose  support is contained in $U_x$ and such that $\langle f , \gamma \rangle \neq 0$.

\begin{proof}
Aiming for a contradiction, suppose $(\gamma_i)_i \subset FS_n(M)$ is a net which weakly converges to some $\gamma \not\in FS_n(M)$. This means that $\supp(\gamma)$ contains at least $n+1$ points $x_1,\ldots,x_{n+1}$. Let $\delta>0$ be small enough so that the balls $B(x_k,\delta)$, for $k=1,\ldots,n+1$, are pairwise disjoint. By \cite[Proposition 2.7]{APPP_2019}, there are $f_k\in\Lip_0(M)$ such that $\supp(f_k)\subset B(x_k,\delta)$ and $\langle f_k , \gamma \rangle \neq 0$. Therefore, if $i$ is large enough we must have $\langle f_k , \gamma_i \rangle \neq 0$ for every $k$, hence $\supp(\gamma_i)\cap B(x_k,\delta)\neq\emptyset$ for every $k$. This is impossible since $\supp(\gamma_i)$ only has $n$ elements.
\end{proof}

\subsection{Quasi-conjugacy } 

It is well known that, hypercyclicity and the notions of topological dynamics introduced in the introduction are preserved under quasi-conjugacy.

\begin{definition}
Let $f : M \rightarrow M$  and  $g: N \rightarrow N$ be two continuous maps
acting on metric spaces $M$ and $N$. The map $f$ is called quasi-conjugate to $g$ if there exists a continuous map $\phi:N\to M$ with dense range such that the diagram 
$$\xymatrix{
    N \ar[r]^g \ar[d]_{\phi}  & N \ar[d]^{\phi} \\
    M \ar[r]_{f} & M
  }$$
commutes, that is, $f\circ\phi=\phi\circ g$. In this case, we say that $\phi$ defines a quasi-conjugacy from $g$ to $f$.  
\end{definition}

\begin{proposition}
Let $(M,d_M)$ and  $(N,d_N)$ be two pointed metric spaces,  let $f : M \rightarrow M$, $g: N \rightarrow N$ be two Lipschitz maps such that $f(0)=0$ and $g(0)=0$, and let  $\phi:N\rightarrow M$ be a Lipschitz map such that $\phi(0)=0$. Then $\phi$ defines a
quasi-conjugacy from $g$ to $f$ if and only if $\widehat{\phi}:\Free(N)\rightarrow\Free(M)$ defines a quasi-conjugacy from $\widehat{g}$ to $\widehat{f}$.
\label{3}
\end{proposition}
\begin{proof}
Thanks to Proposition \ref{denserange}, we have $\phi$ has a dense range if and only if $\widehat{\phi}$ has a dense range. So, it remains to show that $f\circ\phi=\phi\circ g$ is equivalent to $\widehat{f}\circ\widehat{\phi}=\widehat{\phi}\circ\widehat{g}$. Assume that $f\circ\phi=\phi\circ g$, let $\gamma=\sum a_i\delta_N(x_i)\in \lspan\,\delta(N)$, we have
\begin{align*}
    \widehat{\phi}\circ\widehat{g}(\gamma)&=\widehat{\phi}(\sum a_i\delta_N(g(x_i)))\\
    &=\sum a_i\delta_M(\phi\circ g(x_i))\\
    &=\sum a_i\delta_M(f\circ\phi(x_i))\\
    &=\widehat{f}\circ\widehat{\phi}(\sum a_i\delta_N(x_i))\\
    &=\widehat{f}\circ\widehat{\phi}(\gamma),
\end{align*}
so $\widehat{f}\circ\widehat{\phi}=\widehat{\phi}\circ\widehat{g}$ on $\lspan\,\delta(N)$, which implies that $\widehat{f}\circ\widehat{\phi}=\widehat{\phi}\circ\widehat{g}$ on $\Free(N)$. Conversely, suppose that $\widehat{f}\circ\widehat{\phi}=\widehat{\phi}\circ\widehat{g}$, let $x\in N$, we have $\widehat{f}\circ\widehat{\phi}(\delta_N(x))=\widehat{\phi}\circ\widehat{g}(\delta_N(x))$, which is equivalent to $\delta_M(\phi\circ g(x))=\delta_M(f\circ\phi(x))$, since $\delta_M$ is an isometry, we get $\phi\circ g(x)=f\circ\phi(x)$.  
\end{proof}

\section{Hypercyclicity Criterion for Lipschitz operators} \label{section_criteria}

Proving that a given operator is hypercyclic by constructing a hypercyclic vector is not an easy task, it is sometimes easier to check the topological transitivity condition. 
Nonetheless, in many concrete
situations it is not obvious how to verify the latter condition.
The purpose of \textit{the Hypercyclicity Criterion} is to provide several easily verified conditions under which an operator is hypercyclic (actually even weakly mixing). In this section, we will shift those conditions on the Lipschitz maps themselves, which will give us a very useful tool for particular examples.  
Let us start by recalling the statement of the Hypercyclicity Criterion \cite{Survey, GrPe}.
\medskip

\noindent \textbf{The Hypercyclicity Criterion (HC).} Let $X$ be a separable Banach space and let $T : X \to X$ be a bounded linear operator. We will say that $T$ satisfies the HC if there exists an increasing sequence of integers $(n_k)$, two dense sets $X_0$ 
 and $Y_0$ 
 in $X$, and a sequence of maps $S_{n_k} : Y_0 \to X$ such that 
\begin{enumerate}
    \item $T^{n_k}x \to 0$ for any $x \in X_0$;
    \item $S_{n_k}x \to 0$ for any $x \in Y_0$;
    \item $T^{n_k} S_{n_k}y \to y$ for each $y \in Y_0$.
\end{enumerate}
It is well known that if $T$ satisfies the HC then $T$ is hypercyclic (see \cite[Theorem~3.15]{GrPe} e.g.). Moreover, a bounded linear operator satisfies the HC if and only if it is weakly mixing (see \cite[Theorem 2.3]{BesPeris}), and if it satisfies the HC with respect to the full sequence $(n)_{n \in \N}$ then it is actually mixing (see page 32 in \cite{Survey}).
\medskip

The linearization property stated in Proposition~\ref{diagramfree} allows us to formulate a version of the HC for Lipschitz operators  $\widehat{f} : \F(M) \to \F(M)$, only involving metric conditions. 
\begin{theorem}
\label{HCriterion}
Let $(M,d)$ be a pointed separable metric space,  $f:M \to M$ be a Lipschitz map such that $f(0)=0$ and $\lambda\in\mathbb{R}\setminus\lbrace 0\rbrace$. Assume that there exist an increasing sequence of integers $(n_k)_{k\in\mathbb{N}}$, two dense subsets $\mathcal{D}_1$, $\mathcal{D}_2$ in $M$ and a sequence of maps $g_{n_k}:\mathcal{D}_2 \to M$ such that, for any $x\in\mathcal{D}_1$ and $y\in\mathcal{D}_2$ the following conditions hold: 
\begin{enumerate}
    \item $|\lambda|^{n_k}\,d(f^{n_k}(x),0)\underset{k\to+\infty}{\longrightarrow}0$;
    \item $\dfrac{d(g_{n_k}(y),0)}{|\lambda|^{n_k}}\underset{k\to+\infty}{\longrightarrow}0$;
    \item $d(f^{n_k}\circ g_{n_k}(y),y)\underset{k\to+\infty}{\longrightarrow}0$;
\end{enumerate}
Then $\lambda \widehat{f}$ satisfies the Hypercyclicity Criterion. In particular, $\lambda \widehat{f}$ is hypercyclic.
\end{theorem}

\begin{proof}
Let $X_0=\mathrm{span}\,\delta(\mathcal{D}_1)$ and $Y_0=\mathrm{span}\,\delta(\mathcal{D}_2)$. It is clear that $X_0$ and $Y_0$ are dense in $\Free(M)$. Moreover, for every $x\in\mathcal{D}_1$, we have:
\[\|(\lambda\widehat{f})^{n_k}(\delta(x))\|=|\lambda|^{n_k}\|\delta(f^{n_k}(x))\|=|\lambda|^{n_k}\,d(f^{n_k}(x),0)\underset{k\to+\infty}{\longrightarrow}0.\]
Therefore, for every $x_0\in X_0$,  $\|(\lambda\widehat{f})^{n_k}(x_0)\|\underset{k\to+\infty}{\longrightarrow}0$. 

Let $S_{n_k}:Y_0 \to \Free(M)$ be the linear map given by $S_{n_k}(\delta(y))=\dfrac{1}{\lambda^{n_k}}\delta(g_{n_k}(y))$, for every $y\in\mathcal{D}_2$. We thus have \[\|S_{n_k}(\delta(y))\|=\dfrac{d(g_{n_k}(y),0)}{|\lambda|^{n_k}}\underset{k\to+\infty}{\longrightarrow}0\]
and
\[\|(\lambda\widehat{f})^{n_k}\circ S_{n_k}(\delta(y))-\delta(y)\|=\| \delta(f^{n_k}\circ g_{n_k}(y))-\delta(y)\|=d(f^{n_k}\circ g_{n_k}(y),y)\underset{k\to+\infty}{\longrightarrow}0.\]
Therefore, for every $y_0\in Y_0$, $S_{n_k} y_0\underset{k\to+\infty}{\longrightarrow}0$ and $(\lambda\widehat{f})^{n_k}\circ S_{n_k}y_0\underset{k\to+\infty}{\longrightarrow} y_0$. 
\end{proof}

Throughout this paper, we will always have $g_{n_k} = g^{n_k}$ for some function (always denoted by) $g$, and that we will refer to as the ``{\it inverse function}" of $f$ even though $f$ is not a bjiection.
\begin{definition}
We will say that $\widehat{f}$ satisfies the \textit{hypercyclic criterion for Lipschitz operators} (shortened HCL) if $f$ satisfies the conditions of Theorem~\ref{HCriterion} with $\lambda=1$.
\end{definition}

Here we present a setting where the HCL is satisfied.
\begin{proposition}  
Let $(M,d)$ be a complete pointed metric space, and $f:M\to M$ be a Lipschitz map such that $f(0)=0$. Assume that there exist an increasing sequence of integers $(n_k)_{k\in\N}$, a dense subset $D$ in $M$, a subset $J$ of $M$ with $0\in J$, and an integer $p\geq1$ such that:
\begin{enumerate}
    \item For every $x\in D$, $d(f^{pn_k}(x),0)\underset{k\to+\infty}{\longrightarrow}0$;
    \item $f^p_{|J}:J \to M$ is bijective and its inverse is a contraction; 
\end{enumerate}
then $\widehat{f}$ satisfies the HCL. In particular, $\widehat{f}$ is hypercyclic.
\end{proposition}
\begin{proof}
For each $k$, let $g_k=(f_{|J}^{p})^{-n_k}:M \to J$. Since $(f^p_{|J})^{-1}(0)=0$ and $(f^p_{|J})^{-1}$ is a contraction, we get
$d((f^p_{|J})^{-n}(x),0)\underset{n\to+\infty}{\longrightarrow}0$, for each $x\in M$. In particular, for each $x\in M$, $d(g_k(x),0)\underset{k\to+\infty}{\longrightarrow}0$. Moreover, it is clear that the last condition of HCL is satisfied. Hence $\widehat{f}$ is hypercyclic. 
\end{proof}

We now explore the connections between the weakly mixing property and the HCL. This is of course  motivated by the linear case since a bounded operator $T : X \to X$ satisfies the HC if and only if it is weakly mixing \cite[Theorem 3.15]{GrPe}. Unfortunately, if $\widehat{f}$ satisfies the HCL then $f$ is not necessarily weakly mixing (not even transitive; see Example~\ref{ExampleHCnotHCL} and Example~\ref{ExampleHCLnotTransitive}).  Nevertheless the reverse implications holds, we omit the proof since it follows the same line as in \cite[Theorem~3.15]{GrPe}.

\begin{theorem} \label{Thm_WMimpliesHCL}
Let $(M,d)$ be a separable complete pointed metric space without isolated points, and let  $f:M\to M$ be a Lipschitz map such that $f(0)=0$. If $f$ is weakly mixing then $f$ satisfies the HCL.
\end{theorem}

In  \cite{DR}, M. De La Rosa and C. Read gave the first example of hypercyclic operator which  does not satisfy the Hypercyclicity Criterion. In the same direction, we do not know whether there is a Lipschitz operator $\widehat{f}$ which is hypercyclic but fails the Hypercyclicity Criterion. On the other hand, there are Lipschitz operators $\widehat{f}$ satisfying the HC but not the HCL.

\begin{example} \label{ExampleHCnotHCL}
Let $M$ be the compact space $\{0\} \cup \{ \frac{1}{n} \; : \; n \in \N\}$ equipped with the usual distance in $\R$. Let $f$ be the Lipschitz map defined by $f(0)=0$, $f(1)=1$ and $f(\frac{1}{n}) = \frac{1}{n-1}$ for $n \geq 2$.  We claim that $f$ does not satisfy the HCL but satisfies the usual HC.

To simplify the notation, we will write $x_n$ instead of $\frac{1}{n}$.
First, it is readily seen that for every $k \neq 0$, we have that $\lim\limits_{n \to \infty} d(f^n(x_k),0)=1 \neq 0$ and so $f$ does not satisfy the HCL.

Next, it is well known that $\F(M) \equiv \ell_1$. Indeed, the linear operator $\Phi : \ell_1 \to \F(M)$ given by 
$$ \Phi(e_n) = \frac{\delta(x_n) - \delta(x_{n+1})}{d(x_n,x_{n+1})} $$
is a surjective isometry. So $\widehat{f}$ is conjugate to an operator $T :=\Phi^{-1} \circ \widehat{f} \circ \Phi  : \ell_1 \to \ell_1$ so that $T(e_1) = 0$ and  for $n\geq 2$:
\begin{eqnarray*}
T(e_n) &=& \Phi^{-1} \widehat{f} \left(\frac{\delta(x_n) - \delta(x_{n+1})}{d(x_n,x_{n+1})}\right) \\
&=& \Phi^{-1} \left(\frac{\delta(x_{n-1}) - \delta(x_{n})}{d(x_n,x_{n+1})}\right) \\
&=& \frac{d(x_{n-1},x_{n})}{d(x_n,x_{n+1})}\Phi^{-1} \left(\frac{\delta(x_{n-1}) - \delta(x_{n})}{d(x_{n-1},x_{n})}\right) \\
&=& \frac{n+1}{n-1} e_{n-1}.
\end{eqnarray*}
Thus $T$ is a weighted backward shift on $\ell_1$ which is well known to satisfy the HC, so does $\widehat{f}$.
\end{example}

\begin{remark}
Notice that if we consider $1$ to be the base point of $M$ (instead of 0), then $f$ satisfy the HCL. Therefore, when $f$ as multiples fixed points, the choice of the base point matters for studying the HCL. This is clearly in opposition with the other dynamical properties considered in this paper. Nevertheless, the assertion ``there is a choice of the base point such that $f$ satisfies the HCL whenever $\widehat{f}$ satisfies the HC" is false. Indeed, it suffices to only modify the value of $f(1)$ by setting for instance $f(1) = \frac{1}{2}$ to disprove the latter statement ($1$ is not an acceptable base point anymore since it is not a fixed point for the map $f$). 
\end{remark}

\subsection{Application}

Let us first prove the following characterisation of hypercyclicity for ``backward shift Lipschitz operators" defined on (the Lipschitz-free space over) countably branching tree of height one.
\begin{proposition} \label{hyperleftshift}
Let $M= \N \cup \{0\}$ be a pointed metric space and let  $f:M \to  M$ be the map defined by $f(0)=0$ and $f(n)= n-1$ whenever $n \geq 1$. Assume that $M$ is endowed with a metric $d$ such that $f$ is Lipschitz and $d(n,m)=d(n,0)+d(0,m)$ whenever $n \neq m \in \N$.
\begin{enumerate}
\item Then the following conditions are equivalent:
\begin{enumerate}
    \item $\underset{n\to+\infty}{\liminf}\, d(n,0)=0$.
    \item $\widehat{f}$ is hypercyclic. 
\end{enumerate}
\item We also have equivalence between the two stronger conditions:
\begin{enumerate}
    \item $\lim\limits_{n \to \infty} d(n,0) =0$.
    \item $\widehat{f}$ is mixing. 
\end{enumerate}
\end{enumerate}
\end{proposition}

\begin{proof}

According to Proposition~\ref{PropTreesl1}, $$\psi: \delta(n) \in \F(M) \mapsto d(n,0)e_n \in \ell_1$$
induces an onto linear  isometry between $\Free(M)$ and $\ell_1$.  Hence $T=\psi\circ\widehat{f}\circ\psi^{-1} : \ell_1 \to \ell_1$ is quasi-conjugate to $\widehat{f}$ while $Te_n=\frac{d(n-1,0)}{d(n,0)}e_{n-1}$ for $n\geqslant2$ and $Te_1=0$. Thus $T$ is a unilateral weighted backward shift acting on $\ell_1$. According to  \cite[Theorem 1.40]{Survey} (see also Remark 1.41 therein), we can deduce that
\begin{align*}
    \widehat{f} \text{ is hypercyclic } &\iff  T \text{ is hypercyclic } \\
    &\iff \underset{n\to+\infty}{\limsup}\,\dfrac{d(1,0)}{d(2,0)}\times\dfrac{d(2,0)}{d(3,0)}\times\ldots\times\dfrac{d(n-1,0)}{d(n,0)}=+\infty \\
    &\iff  \underset{n\to+\infty}{\liminf}\, d(n,0)=0,
\end{align*}
which proves $(1)$. The proof of assertion $(2)$ is similar (replacing $\limsup$ by $\lim$) and based on the  characterisation of mixing weighted backward shift acting on $\ell_1$; see \cite[Theorem~4.8 and Example~4.9~(a)]{GrPe}.
\end{proof}

As an easy consequence we provide a non-hypercyclic Lipschitz map $f : M \to M$ which satisfies the HCL. Thus, the HCL does not necessarily implies in general that $f$ itself is hypercyclic.  

\begin{example} \label{ExampleHCLnotTransitive}
Let $M$ and $f :M \to M$ be as in the previous proposition with $d(0,n)=\frac{1}{2^n}$ for every $n \in \N$. It is clear that the orbit of any point in $M$ under $f$ is finite, therefore it cannot be dense. Nonetheless, the sequence $(d(n,0))_n$ is decreasing to 0 so that $\widehat{f}$ is mixing.
\end{example}

As another application of Theorem~\ref{HCriterion}, we can state a modified version of \cite[Theorem 3]{DuMoZa_2013}. It corrects the former statement which does not hold in general, as we will see in  Example~$\ref{Counter-example}$. Note that the case when $\widehat{f}$ is the backward shift on (a copy of) $\ell_1$ is included in the following proposition (see again Example $\ref{Counter-example}$).

\begin{proposition}  \label{PropBackShift}
Let $(M,d)$ be a pointed metric space, and let $f:M \to M$ be a Lipschitz map such that $f(0)=0$. Let $(M_n)_{n\geq0}$ be a partition of $M$ such that $M_0= \left\lbrace 0 \right\rbrace$, and for every  $n\geq 1$, $f_{|M_{n+1}}$ is injective, and  $$f(M_{n+1}) = M_n.$$
Let $\lambda\in \mathbb{R}$ be such that $\dfrac{d^+_n}{\lambda^n} \underset{n\to \infty}{\rightarrow} 0$, where $d^+_n = \sup_{x\in M_n} d(x,0)$. Then $\lambda \widehat{f}$ is mixing.
\end{proposition}

\begin{proof}
We apply Theorem $\ref{HCriterion}$ with $\mathcal{D}_1 = \mathcal{D}_2 = M$, $n_k = k$ and $g_k = g^k$ where $g : M \to M$ is the map defined as follow: $g(0) = 0$, and for every $n\geq 1$ and $y\in M_n$, $g(y) = x$ where $x$ is the unique element of $M_{n+1}$ such that $f(x) = y$.\\
Let us check the three conditions of Theorem $\ref{HCriterion}$:
\begin{enumerate}
    \item For every $x\in M$, $f^n(x) = 0$ when $n$ is large enough and hence the condition $(1)$ of Theorem $\ref{HCriterion}$ is satisfied.
    \item Let $k \in \mathbb{N}$ and $y\in M_k$. If $k=0$ and $y=0$ so that $g_n(y) = 0$. If $k\geq 1$ then $g_n(y) \in M_{k + n}$ and we have
    $$ \dfrac{d(g_n(y),0)}{|\lambda|^n} = |\lambda|^k \dfrac{d(g_n(y),0)}{|\lambda|^{k+n}} \leq |\lambda|^k \dfrac{d^+_{k + n}}{|\lambda|^{k+n}} \underset{n \to \infty}{\rightarrow} 0.$$
    \item For every $y\in M$, $f^n \circ g_n(y) = y$ and hence $d(f^n\circ g_n(y),y)=0$.
\end{enumerate}
Thanks to Theorem~\ref{HCriterion}, we conclude that $\lambda \widehat{f}$ satisfies the Hypercyclicity Criterion with respect to the full sequence, so it is mixing.
\end{proof}

We give a counter-example to Proposition~\ref{PropBackShift} in the case when $\frac{d^+_n}{\lambda^n}$ does not converge to $0$. Moreover, this provides an example of a Lipschitz operator $\widehat{f}$ which is supercyclic but not hypercyclic.
\begin{example}\label{Counter-example}
Once more, let $M$ and $f : M \to M$ be as in Proposition~\ref{hyperleftshift} with $d(0,n) = n!$ for every $n \in \N$.
Recall that $\Phi : \delta(x_n) \in \F(M) \mapsto n! \, e_n \in \ell_1$ defines an isometry and $\widehat{f} : \Free(M) \to \Free(M)$ is conjugate to $T : \ell_1 \to \ell_1$ given by
$$T(e_{n+1}) = \dfrac{e_n}{n+1}.$$  
If $\lambda \in \mathbb{R}$ then notice that $\lambda T$ is compact and therefore not hypercyclic. Therefore $\lambda \widehat{f}$ is not hypercyclic as well,  while
$$\left|\dfrac{d^+_n}{\lambda^n} \right| = \dfrac{n!}{|\lambda|^n}\to +\infty.$$  
In order to check that $\widehat{f}$ is supercyclic, it is easy to see that the conditions of Theorem~\ref{LipSCC} below are satisfied.
\end{example}

Next there is a Lipschitz map $f : M \to M$ such that $f$ has an orbit which is dense in $M$ while $\widehat{f}$ is cyclic but it is not supercyclic.
\begin{example} \label{ExamplefTTfhatNOTsuper}
Let $M=\N \cup \{0\}$ be as in Proposition~ \ref{hyperleftshift}, where $(d_n)_n:=(d(0,n))_n$ is decreasing and tending to zero. This time we let $f : M \to M$ be the $1$-Lipschitz map defined by $f(0)=0$ and $f(n)=n+1$ otherwise. Then the orbit of $1$ under $f$ is dense in $M$, and since $\|\widehat{f}\|=\mathrm{Lip}(f) \leqslant1$, we get that $\widehat{f}$ is not hypercyclic. Let us show that $\widehat{f}$ is not even supercyclic. As usual $\widehat{f}$ is conjugate to an bounded operator 
$T : \ell_1 \to \ell_1$ given by $T(e_{n}) = \dfrac{d_{n+1}}{d_n} e_{n+1}$. It is clear that $e_1$ is a cyclic vector for $T$, but $T$ is not supercyclic because it does not have a dense range.
\end{example}

Finally, the following example shows that there is a pointed metric space $M$ and a Lipschitz map $f : M \to M$ such that both $f$ and $\widehat{f}$ are hypercyclic.
\begin{example} 
Let $\sum_2$ be the space of $0$-$1$-sequences, that is, 
$\sum_{2}=\lbrace (x_n)_{n\in\mathbb{N}}:\, x_n\in\lbrace0,1\rbrace\rbrace$. The sequence $(0)_{n \in \N}$ is considered to be the base point of $M$. 
Let $d$ be the metric on $\sum_2$ defined by:
\[d(x,y)=\sum_{n=1}^{+\infty}\dfrac{|x_n-y_n|}{2^n}\]
where $x=(x_1,x_2, \ldots)$ and $y=(y_1,y_2,...)$.
We consider the following map:
\[
    \begin{array}{lrcl}
    \sigma : & \sum_2 & \longrightarrow & \sum_2 \\
        & (x_1,x_2,x_3, \ldots) & \longmapsto & (x_2,x_3, x_4, \ldots) \end{array}
\]
The dynamical system $(\sum_2,\sigma)$ is often called the backward shift on two symbols. Note that $\sigma(0)=0$ and $\sigma$ is $2$-Lipschitz. 
It is well known that $\sigma$ is Devaney choatic (\cite[Theorem 1.36]{GrPe}) and mixing (\cite[Exercice 1.4.1]{GrPe}). Consequently, $\widehat{\sigma}$ is also Devaney choatic and mixing thanks to \cite[Theorem~2.3]{MuPe}. 
\bigskip

We wish to recall that $\sigma$ is quasi-conjugate to a map acting on the unit circle $\mathbb{T}$ be endowed with the normalized distance $d$ defined by:
\[ \forall \; \theta_1,\theta_2 \in [0,1[, \quad 
    d(e^{2\pi i\theta_1},e^{2\pi i\theta_2})=\begin{cases}
   |\theta_1-\theta_2| &\text{ if } |\theta_1-\theta_2| \leqslant\frac{1}{2} \\
    1-|\theta_1-\theta_2| &\text{ if } |\theta_1-\theta_2| \geqslant\frac{1}{2} 
    \end{cases},
\]
Indeed, let $D:\mathbb{T}\rightarrow\mathbb{T}$ be the doubling map $Dz=z^2$. It is known \cite[Example 1.37]{GrPe} that the map
\begin{equation}
        \begin{array}{lrcl}
    \phi : & \sum_2 & \longrightarrow & \mathbb{T} \\
        & (x_n) & \longmapsto & \exp{\left(2\pi i \sum_{n=1}^{\infty}\dfrac{x_n}{2^{n}}\right)} \end{array}
\label{2}        
\end{equation}
defines a quasiconjugacy from $\sigma$ to $D$. Note that $\phi(0)=1$ and $\phi$ is 1-Lipschitz. By Proposition \ref{3}, $\widehat{D}:\Free(\mathbb{T})\to\Free(\mathbb{T)}$ is Devaney choatic and  mixing as well.
\end{example}

\subsection{Some other criteria}

In the same way as we did for the HC, we may also ``push downward" the conditions of other well-known criteria to obtain metric conditions for Lipschitz operators. 
We will quickly mention two examples: the ``Supercyclicity Criterion" (see \cite[Theorem 1.14]{Survey})  and the ``Chaoticity Criterion" (see \cite[Theorem 6.10]{Survey}).

\begin{theorem}[Supercyclicity Criterion for Lipschitz Operators] \label{LipSCC} 
~\\
Let $(M,d)$ be a pointed separable metric space, $f:M\to M$ be a Lipschitz map such that $f(0)=0$. Assume that there exist an increasing sequence of integers $(n_k)_{k\in\mathbb{N}}$, two dense subsets $\mathcal{D}_1$, $\mathcal{D}_2$ in $M$ and a sequence of maps $g_{n_k}:\mathcal{D}_2 \to M$ such that, for any $x\in\mathcal{D}_1$ and $y\in\mathcal{D}_2$ the following conditions hold:
\begin{enumerate}
    \item $d(f^{n_k}(x),0)\, d(g_{n_k}(y),0)\underset{k\to+\infty}{\longrightarrow}0$;
    \item $d(f^{n_k}\circ g_{n_k}(y),y)\underset{k\to+\infty}{\longrightarrow}0$;
\end{enumerate}
Then $\widehat{f}$ is supercyclic.
\end{theorem}

\begin{corollary}
Let $M=\N \cup \{0\}$ be a pointed metric space endowed with a metric $d$ such that the map $f:M\rightarrow M$ defined by $f(0)=0$, $f(1)=0$ and $f(n+1)=n$ is Lipschitz. Then $\widehat{f}$ is supercyclic.
\end{corollary}

\begin{theorem}[A Chaoticity criterion for Lipschitz operators]\label{Lipschitz Chaoticity criterion} 
~\\
Let $(M,d)$ be a pointed separable metric space and let  $f:M\to M$ be a Lipschitz map such that $f(0)=0$. Assume that there exist a dense subset $\mathcal{D}$ in $M$ and a map $g:\mathcal{D} \to \mathcal{D}$ such that, for any $x\in\mathcal{D}$ the following conditions hold:
\begin{enumerate}
    \item $\sum_{n=0}^{\infty}d(f^{n}(x),0)$ and $\sum_{n=0}^{\infty}d(g^{n}(x),0)$ are convergent;
    \item $f\circ g=\mathrm{Id}_{M}$;
\end{enumerate}
Then $\widehat{f}$ satisfies the Chaoticity Criterion.
\end{theorem}

\begin{remark}
 The map $f$ given in  Example \ref{ExampleHCnotHCL} does not satisfy the chaoticity criterion for Lipschitz operators but $\widehat{f}$ satisfies the chaoticity criterion.
 \end{remark}

Similarly as in Proposition~\ref{hyperleftshift}, we characterise below the chaoticity for ``backward shift Lipschitz operators" defined on (the Lipschitz-free space over) countably branching tree of height one.

\begin{proposition}\label{ChaoShift}
Let $M= \N \cup \{0\}$ be a pointed metric space and let  $f:M \to  M$ be the map defined by $f(0)=0$ and $f(n)= n-1$ whenever $n \geq 1$. Assume that $M$ is endowed with a metric $d$ such that $f$ is Lipschitz and $d(n,m)=d(n,0)+d(0,m)$ whenever $n \neq m \in \N$.
Then the following conditions are equivalent:
\begin{enumerate}
\item $\widehat{f}$ is Devaney chaotic. 
\item The series $\sum_{n=1}^{\infty}\, d(n,0)$ converges.
\end{enumerate}
If moreover the sequence $\big( d(n,0)\big)_n$ is decreasing, then the two above conditions are equivalent to
\begin{enumerate}
    \item[(3)] $f$  satisfies the chaoticity criterion for Lipschitz operators.
\end{enumerate}

\end{proposition}
\begin{proof}
As in the proof of Proposition \ref{hyperleftshift}, the operator $\widehat{f}$ is conjugate with the weighted backward  shift operator $T$ defined on  $\ell_1$ by  
$$ Te_1=0 \quad \text{and} \quad \forall n\geqslant 2, \quad   Te_n=\dfrac{d(n-1,0)}{d(n,0)}e_{n-1} .$$ 
So $\widehat{f}$ is chaotic if and only if $T$ is so. Therefore, by \cite[Theorem 6.12]{Survey}, we obtain
\begin{align*}
\widehat{f} \text{  chaotic }&\Longleftrightarrow \sum_{n\geqslant1}\left(\dfrac{d(1,0)}{d(n,0)}\right)^{-1}<+\infty\\
&\Longleftrightarrow \sum_{n=1}^{\infty}\, d(n,0)<+\infty,
\end{align*}
which proves  $(1)\Longleftrightarrow(2)$.
It is clear that $(3)\Rightarrow(1)$, so let us show that $(2)\Rightarrow(3)$. Let $g:M\rightarrow M$ be the map defined by $g(0)=0$ and $g(n)=n+1$. Fix $m\geqslant1$. Since $d(g^n(m),0)\leqslant d(n,0)$ and $\sum_{n=1}^{\infty}\, d(n,0)$ converges, we have that the series $\sum_{n=1}^{\infty}\, d(g^n(m),0)$ also converges. Moreover, we have $f\circ g=\mathrm{Id}_M$ and the series $\sum_{n=1}^{\infty}\, d(f^n(m),0)$ converges as well.  
\end{proof}

Recall that we proved in Proposition~\ref{PropPeriodic} that if $\mathrm{Per}(f)$ is dense in $M$, then $\mathrm{Per}(\widehat{f})$ is dense in $\F(M)$. We claim that the reverse implication is not true.
\begin{example} \label{ExamplePeriodic}
Let $M$ and $f: M \to M$ be as in Proposition~\ref{ChaoShift}, with $d(0,n)=\frac{1}{2^n}$ for every $n \geq 1$.
Clearly $\mathrm{Per}(f)=\lbrace0\rbrace$ is not dense in $M$. On the other hand, since
$$\sum_{n=1}^{\infty}d(n,0)=\sum_{n=1}^{\infty}\dfrac{1}{2^n}<+\infty,$$
 $\widehat{f}$ is Devaney chaotic thanks to Proposition~\ref{ChaoShift}.
\end{example}

\section{The particular case of compact intervals}

From now on, we will consider metric spaces $M= [a,b]$, where $a<b\in \mathbb{R}$, and $f : M \to M$ will be a Lipschitz map having (at least) one fixed point $c\in [a,b]$.
For such metric spaces $M$, we have a surjective isometry $\Phi : \delta(x) \in \F(M) \mapsto \mathbf{1}_{[c,x]} \in L^1([a,b])$ (where $\mathbf{1}_{[c,x]}$ is understood as $-\mathbf{1}_{[x,c]}$ when $x\leq c$). Thus, $\widehat{f} : \F(M) \to \F(M)$ is conjugate to an operator $T : L^1([a,b]) \to L^1([a,b])$. This operator acts on indicator functions as follows: if $a\leq s \leq t \leq b$, we have
\begin{align}\label{formulesurL1}
T(\mathbf{1}_{[s,t]})=\left\{\begin{array}{cl}
\mathbf{1}_{[f(s),f(t)]},& \text{if} \ \ f(s) \leq f(t) \medskip\\
-\mathbf{1}_{[f(t),f(s)]}, &  \text{if} \ \ f(t) \leq f(s).
\end{array}\right.
\end{align} 
The next theorem will give sufficient conditions on $f$ which ensure that $\widehat{f}$ is hypercyclic, and so this will exhibit a new class of hypercyclic operators acting on $L^1([a,b])$.
To our knowledge, there is not much study on hypercyclic operators defined on $L^1(I)$ when $I$ is a bounded interval in $\R$. If $I= [0, +\infty)$ or $I = \mathbb{R}$, operators of translation $T_t : f \mapsto f(\cdot + t)$ have been studied as well as the dynamical properties of the $C_0$-semigroup $(T_t)_{t\geq 0}$ on weighted $L^p$-spaces $L^p(I, \omega \text{d}\lambda)$, where $\lambda$ is the Lebesgue measure and $\omega$ is a weight on $I$. For instance, in
\cite{DSW}, necessary and sufficient conditions are given in terms of the weight $\omega$ that ensure the hypercyclicity of $T_t$ or $(T_t)_{t\geq 0}$.  

\medskip

\begin{theorem}\label{mainthminterval}
If $f:[a,b]\to[a,b]$ is a Lipschitz and topologically transitive map with a fixed point $c\in[a,b]$, then $\widehat{f}$ is weakly mixing. \\
If moreover $f$ admits at least two fixed points, then $\widehat{f}$ is mixing.
\end{theorem}

\begin{proof}
Our main ingredient is a result due to Barge and Martin (see \cite[Theorem~3]{Barge}), the next formulation can be found in  \cite[Theorem~2.19]{Sy}. It states that, since $f$ is topologically transitive, one of the following cases holds:
\begin{enumerate}
    \item[$(i)$] $f$ is mixing.
    \item[$(ii)$] $c \in (a,b)$ is the unique fixed point of $f$, $f([a,c])=[c,b]$, $f([c,b])=[a,c]$ and both maps $f^{2}_{|[a,c]}$, $f^{2}_{|[c,b]}$ are mixing.
\end{enumerate}
If $f$ has at least two fixed points then it is mixing, and $\widehat{f}$ is also mixing thanks to \cite[Theorem~2.3]{MuPe}. 
Otherwise, $c \in (a,b)$ is the unique fixed point of $f$ and $f([a,c])=[c,b]$, $f([c,b])=[a,c]$ and both maps $f^{2}_{|[a,c]}$, $f^{2}_{|[c,b]}$ are mixing. Let $h_a=f^{2}_{|[a,c]}$ and $h_b=f^{2}_{|[c,b]}$. By the latter, we have that $$\widehat{h}_a:\F([a,c]) \to \F([a,c]) \quad \text{ and } \quad \widehat{h}_b:\F([c,b]) \to \F([c,b])$$ are mixing. It is enough to show that $\widehat{f
^2}$ is mixing. 
To do so, we will show that the corresponding conjugate map associated to $\widehat{f^2}$ (see \eqref{formulesurL1}), say $T: L^1([a,b]) \to L^1([a,b])$, is mixing. We will consider $L^1([a,c])$ and $L^1([c,b])$ as subspaces of $L^1([a,b])$.  Let $B(u,r), B(v,r') \subset L^1([a,b])$ be two open balls and write $u = u_1 + u_2, v=v_1 + v_2$ where $u_1, v_1 \in L^1([a,c])$ and $u_2, v_2 \in L^1([c,b])$. The operator $T$ maps $L^1([a,c])$ into $L^1([a,c])$ and $L^1([c,b])$ into $L^1([c,b])$. Its restrictions to these two spaces, that we will denote by $T_1$ and $T_2$ respectively, are mixing. Indeed, $T_1$ is conjugate to $\widehat{h_a}$ while $T_2$ is conjugate to $\widehat{h_b}$. Hence, there exists $N\geq 1$ such that for any $n\geq N$,
$$T_1^n(B(u_1, r/2)) \cap B(v_1, r'/2) \neq \emptyset  \quad \text{ and } \quad T_2^n(B(u_2, r/2)) \cap B(v_2, r'/2)\neq \emptyset.$$
Then, for every $n\geq N$, there exist $u_{1,n} \in B(u_1, r/2)$ and $u_{2,n} \in B(u_2, r/2)$ such that
$$T_1^n(u_{1,n}) \in B(v_1, r'/2) \quad \text{ and } \quad T_2^n(u_{2,n}) \in B(v_2, r'/2).$$
We deduce that $u_n := u_{1,n} + u_{2,n} \in B(u, r)$, where we naturally see $L^1([a,c])$ and $L^1([c,b])$ as subspaces of $L^1([a,b])$. Moreover, we have 
$$T^n(u_n) = T_1^n(u_{1,n}) + T_2^n(u_{2,n}) \in B(v_1, r'/2) + B(v_2, r'/2) \subset B(v, r').$$
We proved that for every $n\geq N$, $T^n(B(u, r)) \cap B(v, r') \neq \emptyset$, which shows that $T$ is mixing.
\medskip
\end{proof}

\begin{remark}
If $c=a$ or $c=b$, then necessarily $f$ has another fixed point. Indeed, assume that $c=a$, the case $c=b$ being similar. The map $f$ is continuous and transitive, so it must be onto. In particular, there exists $y\in (a,b]$ such that $f(y)=b$. Hence, the continuous function $h(x)=f(x)-x$ satisfies $h(y)\geq 0$ and $h(b)\leq 0$ so it vanishes at some point which is a fixed point of $f$, distinct from $a$. 
\end{remark}

\begin{corollary} \label{CorDC}
Let $f:[a,b] \to [a,b]$ be a Lipschitz and topologically transitive map with a fixed point $c\in[a,b]$. Then $\widehat{f}$ is Devaney chaotic.
\end{corollary}

\begin{proof}
It is known \cite{VeBe} that a continuous map $g : I \to I$, where $I$ is a real interval, is topologically transitive if and only if it is Devaney chaotic. 
By assumption, our Lipschitz map  $f:[a,b]\to [a,b]$ is transitive, so it is Devaney chaotic as well. By Proposition~\ref{PropPeriodic}, the set of periodic points of $\widehat{f}$ is dense in $\F(M)$. Moreover $\widehat{f}$ is hypercyclic thanks to Theorem~\ref{mainthminterval}. Therefore, $\widehat{f}$ is Devaney chaotic.
\end{proof}

\begin{remark}
We do not know whether $\widehat{f} : \F(\R) \to \F(\R)$ is hypercyclic whenever $f : \R \to \R$ is a transitive Lipschitz map (having a fixed point). In fact, the decomposition given by \cite[Theorem~2.19]{Sy} was crucial to our proof (see also \cite[Theorem 2.2]{splitingtheorem} for a similar statement in the more general setting of locally connected compact metric spaces). We do not  know if a comparable result holds in the case of unbounded intervals. Nevertheless, any weakly mixing Lipschitz maps $f : \R \to \R$ produces a weakly mixing Lipschitz operator acting on $L^1(\R)$ thanks to \cite[Theorem~2.3]{MuPe} (see \cite{Nagar} at page 2 for an example).
\end{remark}

\subsection{Application}

In this subsection we will provide two Lipschitz self-maps $f$ defined on $[0,1]$ which illustrate the two cases in \cite[Theorem~2.19]{Sy}. 
They are hypercyclic so that their linearization are both  mixing and Devaney chaotic, thanks to Theorem~\ref{mainthminterval}. Also, the obtained operators $\widehat{f}$ acting on $L_1([0,1])$ will be made explicit.

\begin{example}\label{Example1interval}
The map $f : [0,1] \to [0,1]$ defined below is often called the tent map, it is transitive as it is explained in \cite[Examples 1.12 ~(a)]{GrPe}. Here we consider $0$ to be the base point of $[0,1]$. 

\begin{minipage}{0.50\linewidth}
\begin{equation*}
f(x) =
     \left\{
        \begin{array}{rl}
        2x & \text{if } 0 \leq x \leq \frac 12,   \\[0.5em]
        2-2x & \text{if }  \frac 12 \leq x \leq 1.
        \end{array}
     \right.
\end{equation*}
\end{minipage}
\begin{minipage}[c]{0.45\linewidth}
\begin{tikzpicture}[thick, scale=2.3]

    \draw[->,>=latex, gray] (-0.2,0) -- (1.2,0) node[above] {$x$};
	\draw[->,>=latex, gray] (0,-0.2) -- (0,1.2) node[right] {$y$};

	\draw[domain=0:0.5, blue,ultra thick, smooth] plot ({\x},{2*\x}) node[above] {$f(x)$};
	\draw[domain=0.5:1, blue,ultra thick, smooth] plot ({\x},{-2*\x+2}) ;

	\fill (0,0) circle (0.2pt) node[below left] {$0$};
	\fill (1,0) circle (0.2pt) node[below] {$1$};
	\fill (0.5,0) circle (0.2pt) node[below] {$\frac 12$};
	\fill (0.25,0) circle (0.2pt) node[below] {$\frac 14$};
	\fill (0.75,0) circle (0.2pt) node[below] {$\frac 34$};
	\fill (0,1) circle (0.2pt) node[left] {$1$};
	\fill (0,0.5) circle (0.2pt) node[left] {$\frac 12$};
	\fill (0,0.25) circle (0.2pt) node[left] {$\frac 14$};
	\fill (0,0.75) circle (0.2pt) node[left] {$\frac 34$};

\end{tikzpicture}
\end{minipage}
\smallskip

Consequently $\widehat{f}$ is mixing thanks to Theroem~\ref{mainthminterval}. The latter fact can also be checked by showing that $\widehat{f}$ satisfies the HCL with respect to the full sequence, where $g(x)=\frac{x}{2}$ stands for the ``inverse" mapping. We also know that $\widehat{f}$ is Devaney chaotic thanks to Corollary~\ref{CorDC}.
As explained at the beginning of this section, $\widehat{f}$ is conjugate to an operator $T : L^1([0,1]) \to L^1([0,1])$ which maps any indicator function $\mathbf{1}_{[a,b]}, a<b$, to
\begin{equation*}
T(\mathbf{1}_{[a,b]}) =
     \left\{
        \begin{array}{rl}
        \mathbf{1}_{[2a,2b]} & \text{if } b\leq \frac{1}{2},   \\[0.5em]
        -\mathbf{1}_{[-2b+2,-2a+2]} & \text{if }  a\geq \frac{1}{2}.
        \end{array}
     \right.
\end{equation*}
One can check that  $T$ acts as a kind of backwards shift on the haar basis $(h_m)_{m}$ of $L^1([0,1])$ (more details will be given in the second example). Of course, this is not always the case: for instance the tent map is conjugate to the ``logistic map" $L(x) = 4x(1-x)$ (see \cite[Example~1.6]{GrPe}), and the operator $T' : L^1([0,1]) \to L^1([0,1]) $ associated to $\widehat{L}$ does not necessarily send a haar function to another. 
\end{example}

We now study another well-known example.
It is a Lipschitz map $f : [0,1] \to [0,1]$, with only one fixed point $1/2$, and such that $f$ is  Devaney chaotic but not weakly mixing. 

\begin{example} \label{Example2}
Let us consider the following Lipschitz map:
\medskip

\begin{minipage}{0.50\linewidth}
\begin{equation*}
f(x) =
     \left\{
        \begin{array}{rl}
        \frac 12 +2x & \text{if } 0 \leq x \leq \frac 14,   \\[0.5em]
         \frac 32 -2x & \text{if }  \frac 14 \leq x \leq  \frac 12 ,\\[0.5em]
         1-x & \text{if } \frac 12 \leq x \leq 1.
        \end{array}
     \right.
\end{equation*}
\end{minipage}
\begin{minipage}[c]{0.45\linewidth}
\begin{tikzpicture}[thick, scale=2.3]

    \draw[->,>=latex, gray] (-0.2,0) -- (1.2,0) node[above] {$x$};
	\draw[->,>=latex, gray] (0,-0.2) -- (0,1.2) node[right] {$y$};

	\draw[domain=0:0.25, blue,ultra thick, smooth] plot ({\x},{2*\x+0.5}) node[above, right] {$f(x)$};
	\draw[domain=0.25:0.5, blue,ultra thick, smooth] plot ({\x},{-2*\x+3/2}) ;
	\draw[domain=0.5:1, blue,ultra thick, smooth] plot ({\x},{1-\x}) ;

	\fill (0,0) circle (0.2pt) node[below left] {$0$};
	\fill (1,0) circle (0.2pt) node[below] {$1$};
	\fill (0.5,0) circle (0.2pt) node[below] {$\frac 12$};
	\fill (0.25,0) circle (0.2pt) node[below] {$\frac 14$};
	\fill (0.75,0) circle (0.2pt) node[below] {$\frac 34$};
	\fill (0,1) circle (0.2pt) node[left] {$1$};
	\fill (0,0.5) circle (0.2pt) node[left] {$\frac 12$};
	\fill (0,0.25) circle (0.2pt) node[left] {$\frac 14$};
	\fill (0,0.75) circle (0.2pt) node[left] {$\frac 34$};

\end{tikzpicture}
\end{minipage}
\smallskip

Note that here the metric space that we consider is $M=[0,1]$ with $\frac 12$ being the distinguished point of $M$. In what follows, it will be quite convenient to consider the second iterated of $f$, so we give its definition explicitly bellow.
\medskip

\begin{minipage}{0.50\linewidth}
\begin{equation*}
f^2(x) =
     \left\{
        \begin{array}{rl}
         \frac 12 -2x  & \text{if } 0 \leq x \leq \frac 14  , \\[0.5em]
         -\frac 12 + 2x & \text{if }  \frac 14 \leq x \leq  \frac 34 ,\\[0.5em]
         \frac 52 - 2x & \text{if } \frac 34 \leq x \leq 1.
        \end{array}
     \right.
\end{equation*}
\end{minipage}
\begin{minipage}[c]{0.45\linewidth}
\begin{tikzpicture}[thick, scale=2.3]

    \draw[->,>=latex, gray] (-0.2,0) -- (1.2,0) node[above] {$x$};
	\draw[->,>=latex, gray] (0,-0.2) -- (0,1.2) node[right] {$y$};

	\draw[domain=0:0.25, blue,ultra thick, smooth] plot ({\x},{-2*\x+0.5});
	\draw[domain=0.25:0.75, blue,ultra thick, smooth] plot ({\x},{2*\x-1/2})  node[above] {$f^2(x)$};
	\draw[domain=0.75:1, blue,ultra thick, smooth] plot ({\x},{5/2-2*\x}) ;

	\fill (0,0) circle (0.2pt) node[below left] {$0$};
	\fill (1,0) circle (0.2pt) node[below] {$1$};
	\fill (0.5,0) circle (0.2pt) node[below] {$\frac 12$};
	\fill (0.25,0) circle (0.2pt) node[below] {$\frac 14$};
	\fill (0.75,0) circle (0.2pt) node[below] {$\frac 34$};
	\fill (0,1) circle (0.2pt) node[left] {$1$};
	\fill (0,0.5) circle (0.2pt) node[left] {$\frac 12$};
	\fill (0,0.25) circle (0.2pt) node[left] {$\frac 14$};
	\fill (0,0.75) circle (0.2pt) node[left] {$\frac 34$};

\end{tikzpicture}
\end{minipage}
\medskip

We then have the following claims with respect to $f$ and $\widehat{f}$.
\medskip

\begin{enumerate}[$(i)$, itemsep=5pt]
    \item \textit{The map $f$ is Devaney chaotic, but not weakly mixing} (see \cite[Example~2.21]{Sy}).

    \item \textit{The map $f^2$ is not topologically transitive.} Indeed, notice that  $f^2([0,\frac 12]) \subset [0 , \frac 12]$ and $f^2([\frac 12 , 1]) \subset [\frac 12 , 1]$.

    \item \textit{The map $\widehat{f^2}$ satisfies the HCL.} Indeed, we claim that $\widehat{f^2}$ has the HCL where the ``inverse" map $g$ is defined for every $x \in [0,1]$, by $g(x) = \frac x2 + \frac 14$.
    
    \item \textit{The operator $\widehat{f}$ is mixing and Devaney chaotic.} First, since $f$ is transitive, notice that we can deduce that $\widehat{f}$ is weakly mixing and Devaney chaotic by applying Theorem~\ref{mainthminterval} and Corollary~\ref{CorDC}. Furthermore, one can prove directly that actually $\widehat{f}$ satisfies the  HCL condition with respect to the full sequence, where the ``inverse" map $g$ is defined by 
$$ g(x) = \left\{ \begin{array}{cc}
     1-x & \text{ if } 0 \leq x < \frac 12,  \\
     -\frac 12 x + \frac 34 & \text{ if } \frac 12 < x \leq 1.
\end{array}\right. $$
We consider the Dyadic numbers $\mathcal{D}$ in $[0,1]$ as the dense subset of $M$ on which one has to check the conditions given by the HCL.
Since we do not pass to a subsequence for proving those conditions,  $\widehat{f}$ is mixing.
\end{enumerate}

\begin{remark}
Since $g = f^2$ is not topologically transitive but $\widehat{g}$ is Devaney chaotic and mixing, the implication ``$\widehat{g}$ Devaney chaotic and mixing $\implies$ $g$ transitive" does not hold in general.
\end{remark}

We now descirbe the operator $S : L^1([0,1]) \to L^1([0,1])$ conjugate to $\widehat{f}$ by its action on the Haar basis $(h_m)_m$. 
Recall that the Haar functions $(h_m)_{m\geq 0}$ are given by $h_1 = \mathbf{1}_{[0,1]}$ and for every $n\geq 0$ and every $1\leq k \leq 2^n$:
$$h_{2^n+k}(t) = \mathbf{1}_{\left[\frac{2k-2}{2^{n+1}},\frac{2k-1}{2^{n+1}}\right]}(t) -  \mathbf{1}_{\left[\frac{2k-1}{2^{n+1}},\frac{2k}{2^{n+1}}\right]}(t), \ \ t\in [0,1].$$
It is well-known that $(h_m)_{m\geq 0}$ is a Schauder basis of $L^1([0,1])$. 
We have
$$S(h_1) = S(\mathbf{1}_{[0,1]}) = -\mathbf{1}_{[f(1),f(0)]} = -\mathbf{1}_{\left[0,\frac{1}{2}\right]},$$
$$S(h_2)=\mathbf{1}_{\left[0,\frac{1}{2}\right]}, ~ S(h_3)=2\mathbf{1}_{\left[\frac{1}{2}, 1\right]}, ~ S(h_4)=\mathbf{1}_{\left[0,\frac{1}{4}\right]} - \mathbf{1}_{\left[\frac{1}{4},\frac{1}{2}\right]}.$$
Now let $n\geq 2$. We distinguish three cases:\\
$\bullet$ If $0\leq k \leq 2^{n-2}$, the two intervals defining $h_{2^n+k}$ are included in $\left[0, \frac{1}{4} \right]$ so that
\begin{align*}
S(h_{2^n+k})
& = \mathbf{1}_{\left[\frac{1}{2}-\frac{2k-2}{2^{n}},\frac{1}{2}-\frac{2k-1}{2^{n}}\right]} - \mathbf{1}_{\left[\frac{1}{2}-\frac{2k-2}{2^{n}},\frac{1}{2}-\frac{2k-1}{2^{n}}\right]} \\
& = \mathbf{1}_{\left[\frac{2(2^{n-1}+k)-2}{2^{n}},\frac{2(2^{n-1}+k)-1}{2^{n}}\right]} - \mathbf{1}_{\left[\frac{2(2^{n-1}+k)-1}{2^{n}},\frac{2(2^{n-1}+k)}{2^{n}}\right]} \\
& = h_{2^{n-1}+2^{n-2}+k}.
\end{align*}
$\bullet$ If $2^{n-2}+1\leq k \leq 2^{n-1}$, the two intervals that define $h_{2^n+k}$ are included in $\left[\frac{1}{4}, \frac{1}{2} \right]$ and with similar computations we get
$$
S(h_{2^n+k}) = h_{2^{n-1}+3.2^{n-2}-k+1}.
$$
$\bullet$ Finally, if $2^{n-1}+1\leq k \leq 2^n$, both intervals that define $h_{2^n+k}$ are included in $\left[\frac{1}{2}, 1 \right]$ and we have
$$
S(h_{2^n+k}) = h_{2^n+2^n-k+1}.
$$
\end{example}

\subsection{An extension to some compact \texorpdfstring{$\R$-trees}{R-trees}}

Our Theorem~\ref{mainthminterval} can actually be extended to a more general setting.
Recall that an $\R$-tree
is an arc-connected metric space $(M,d)$ with the property that there is a unique arc connecting any pair of points $x\neq y\in M$ and it moreover is isometric to the real segment $[0,d(x,y)]\subset\R$.
A point $x\in M$ is called a branching point if $M\setminus\set{x}$ has at least three connected components; we let $\mathrm{Br}(M)$ the set of all branching point of $M$. The main ingredient for proving Theorem~\ref{mainthminterval} was the decomposition given by \cite[Theorem~2.19]{Sy}. A similar result actually holds for compact trees $M$ such that $M \setminus Br(M)$ has finitely many connected components: If $f : M \to M$ is a continuous and transitive map, then
\begin{itemize}
    \item Either $f$ is mixing,
    \item Or there is a positive integer $n_0$ such that there are an interior fixed point $c$ and subtrees $M_1,\ldots,M_{n_0}$ of $M$ with $\cup_i M_i = M$, $M_i \cap M_j = \{c\}$ whenever $i \neq j$ and $f(M_i) =M_{i+1 (\mathrm{mod} \; n_0)}$ for $1 \leq  i\leq n_0$. Moreover, $f^{n_0}\restriction_{M_i}$ is mixing for every $1 \leq i \leq n_0$.
\end{itemize}
The previous statement can be found in \cite[Proposition~2.6]{Wang2011} and it is based on \cite[Proposition 3.1]{Alseda}.
Therefore, we can state the following (the proof is similar to the one of Theorem~\ref{mainthminterval} and left to the reader).

\begin{theorem} \label{ThmRtrees}
Let $M$ be a compact $\R$-tree such that $M \setminus Br(M)$ has finitely many connected components, and let $f:M\to M$ be a Lipschitz and topologically transitive map with a fixed point $c \in M$. Then either $\widehat{f}$ is mixing or there exists $n_0 \in \mathbb N$ such that $\widehat{f}^{n_0}$ is mixing. In any case, $\widehat{f}$
 is weakly mixing.
\end{theorem}

It is proved in \cite{carac} that a separable metric space $(M,d)$ is locally arcwise connected and uniquely arcwise connected if and only if it admits an equivalent metric $d'$ such that $(M,d')$ is an $\R$-tree. Since $\F(M,d)$ and $\F(M,d')$ are isomorphic, the previous theorem therefore applies to an even more general class of spaces.
In the same way we can also state the following extension of Corollary~\ref{CorDC}.
\begin{corollary}
Let $M$ be a compact $\R$-tree such that $M \setminus Br(M)$ has finitely many connected components. Let $f:M\to M$ be a Lipschitz and topologically transitive map with a fixed point $c \in M$. Then $\widehat{f}$ is Devaney chaotic
\end{corollary}
\begin{proof}
It is proved in \cite[Lemma 2.3]{Wang2011} that $f : M \to M$ is topologically transitive if and only if it is Devaney chaotic. Therefore, exactly as in the proof of Corollary~\ref{CorDC}, we use Proposition~\ref{PropPeriodic} as well as Theorem~\ref{ThmRtrees} to conclude.
\end{proof}

\addtocontents{toc}{\protect\setcounter{tocdepth}{0}}
\section*{Acknowledgments}

Part of this work was carried out when the second named author was working in the ``Laboratoire d'analyse et de math\'ematiques appliqu\'ees" in Marne-la-Vall\'ee.
He is deeply grateful for the excellent working conditions there. 
The authors would also like to thank Romuald Ernst and Anton\'in Proch\'azka for useful conversations, as well as Evgeny Abakumov and St\'ephane Charpentier for valuable comments which improved the presentation of this paper.

% BIBLIOGRAPHY


\begin{thebibliography}{99}


\bibitem{AlbiacKalton}  F. Albiac and N.J. Kalton, \textit{Lipschitz structure of quasi-Banach spaces}, Israel J. Math. 170 (2009), 317--335. 

\bibitem{AP20} R.J. Aliaga and E. Perneck\'a, \textit{Supports and extreme points in Lipschitz-free spaces}, Rev. Mat. Iberoam. (2020), in press, arXiv:1810.11278.

\bibitem{APPP_2019}
	R. J. Aliaga, E. Perneck\'a, C. Petitjean and A. Proch{\'{a}}zka,
	\emph{Supports in Lipschitz-free spaces and applications to extremal structure},
	arXiv:1909.08843.

\bibitem{AlPePr_2019}
	R. J. Aliaga, C. Petitjean and A. Proch{\'{a}}zka,
	\emph{Embeddings of Lipschitz-free spaces into $\ell_1$},
	arXiv:1909.05285.
	
	
\bibitem{splitingtheorem} Ll. Alsed\`a,  M. A. del Río and J. A. Rodríguez, 
\textit{A splitting theorem for transitive maps},
J. Math. Anal. Appl. 232 (1999), no. 2, 359--375. 	
	
	
\bibitem{Alseda}	
	Ll. Alsed\`a, S. Kolyada, J. Llibre and \v{L}. Snoha,
Entropy and periodic points for transitive maps,
Trans. Amer. Math. Soc. 351 (1999), no. 4, 1551--1573. 
	
\bibitem{Ansari} S. I. Ansari, 
\textit{Existence of hypercyclic operators on topological vector spaces},
Journal of Functional Analysis, 148 (1997), no. 2, 384--390. 	
	
	
\bibitem{Barge} M. Barge and J. Martin,
\textit{Chaos, periodicity, and snakelike continua},
Trans. Amer. Math. Soc. 289 (1985), no. 1, 355--365. 
	
\bibitem{Survey}
F. Bayart and E. Matheron,
\emph{Dynamics of linear operators},
Cambridge Tracts In Mathematics, 179 (2009), xiv+337. 

\bibitem{BesPeris}
J. B\`es and A. Peris,
\emph{Hereditarily Hypercyclic Operators}, J. Funct. Anal. 167 (1999) 94--112.

\bibitem{Bourgain} J. Bourgain,
\textit{Real isomorphic complex Banach spaces need not be complex isomorphic},
Proc. Amer. Math. Soc. 96 (1986), no. 2, 221--226. 

\bibitem{vargas2} M. G. Cabrera-Padilla and A. Jim\'enez-Vargas,\textit{ A new approach on Lipschitz compact operators},
Topology Appl. 203 (2016), 22--31.


\bibitem{DR}
	M. De La Rosa and C. Read, 
	\emph{A hypercyclic operator whose direct sum $T\oplus T$ is not hypercyclic},
	Journal of Operator Theory, 61 (2009), no. 2, 369--380.
	
\bibitem{DSW}
W. Desch, W. Schappacher and G. F. Webb,
	\emph{Hypercyclic and chaotic semigroups of linear operators},
	Ergod. Th. and Dynam. Sys. 1997), 17, 793--819.


\bibitem{DuMoZa_2013}
	M. V. Dubey, Z. H. Mozhyrovska, A. V. Zagorodnyuk,
	\emph{Hypercyclic Operators on Lipschitz Spaces},
	Matematychni Studii. V.39, No.1.


\bibitem{Feldman} N.S. Feldman, \textit{Linear chaos}, preprint, (2001),  available  at\\ \href{https://feldman.academic.wlu.edu/files/pdffiles/LinearChaos.pdf}{https://feldman.academic.wlu.edu/files/pdffiles/LinearChaos.pdf}.

\bibitem{Godard_2010}
	A. Godard,
	\emph{Tree metrics and their Lipschitz-free spaces},
	Proc. Amer. Math. Soc. 138 (2010), no. 12, 4311--4320.

\bibitem{GoKa_2003}
  G. Godefroy and N. J. Kalton,
	\emph{Lipschitz-free Banach spaces},
  Studia Math. 159 (2003), 121--141.
 

\bibitem{GrPe} K.G. Grosse-Erdmann and A. Peris,
\emph{ Linear Chaos,}
Universitext Series. Springer. (2011).
     
 
\bibitem{vargas} A. Jim\'enez-Vargas, J.M. Sepulcre and M. Villegas-Vallecillos,
\textit{Lipschitz compact operators,}
J. Math. Anal. Appl. 415 (2014), no. 2, 889--901.
 
 \bibitem{Kalton04} N. J. Kalton, \textit{Spaces of Lipschitz and Hölder functions and their applications}, Collect. Math. 55 (2004), no. 2, 171–217.


\bibitem{carac} J. C. Mayer, L. K. Mohler, L. G. Oversteegen and E. D. Tymchatyn, 
\textit{Characterization of separable metric R-trees},
Proc. Amer. Math. Soc. 115 (1992), no. 1, 257--264. 
 
 \bibitem{MuPe}
  M. Murillo-Arcila and A. Peris,
	\textit{Chaotic behaviour on invariant sets of linear operators},
  Integral Equations and Operator Theory, 81 (2015), 483--497.  
 
 \bibitem{Nagar} A. Nagar and S. P.Sesha Sai,
\textit{Some classes of transitive maps on $\R$},
J. Anal. 8 (2000), 103--111. 


\bibitem{RosendalTalk} C. Rosendal, \textit{Two applications of Arens-Eells spaces to geometric group theory and abstract harmonic analysis}, presentation at the ``Banach spaces webinars", June 19 (2020), available at \href{http://www.math.unt.edu/~bunyamin/banach}{http://www.math.unt.edu/~bunyamin/banach}.  
 
\bibitem{Sy}
   S. Ruette,
	\emph{Chaos on the Interval},
  American Mathematical Soc., (2017)

  
  \bibitem{VeBe}
  M. Vellekoop, R. Berglund,
	\emph{On Intervals, Transitivity = Chaos},
  The American Mathematical Monthly, 101 (1994), 353-355. 
  
\bibitem{Wang2011}  H. Wang and H. Fu,
\emph{Some remarks on chaos in topological dynamics},
Appl. Gen. Topol. 12 (2011), no. 2, 95–100.
  
\bibitem{Weaver2}
  N. Weaver,
	\emph{Lipschitz algebras}, 2nd ed.,
	World Scientific Publishing Co., River Edge, NJ, 2018.


\end{thebibliography}
\end{document}